\documentclass[12pt,a4paper,twoside]{amsart}

\usepackage[english]{babel}
\usepackage[latin1]{inputenc}
\usepackage[T1]{fontenc}

\usepackage{enumerate}

\allowdisplaybreaks

\usepackage[DIV=14,BCOR=2cm]{typearea}

\usepackage{hyperref}

\newcommand{\IR}{\mathbb{R}}
\newcommand{\IN}{\mathbb{N}}
\newcommand{\IZ}{\mathbb{Z}}
\newcommand{\IC}{\mathbb{C}}
\newcommand{\IQ}{\mathbb{Q}}

\newcommand{\IH}{\mathbb{H}}
\newcommand{\frake}{\mathfrak{e}}
\newcommand{\frakv}{\mathfrak{v}}
\newcommand{\calE}{\mathcal{E}}
\newcommand{\calF}{\mathcal{F}}
\newcommand{\calG}{\mathcal{G}}

\newcommand{\calK}{\mathcal{K}}

\newcommand{\calW}{\mathcal{W}}
\renewcommand{\Re}{\operatorname{Re}}
\renewcommand{\Im}{\operatorname{Im}}

\newcommand{\Mp}{\operatorname{Mp}}
\newcommand{\CT}{\operatorname{CT}}

\newcommand{\Gr}{\operatorname{Gr}}
\newcommand{\SL}{\operatorname{SL}}

\newcommand{\bs}{\ensuremath{\backslash}}
\newcommand{\Iso}{\operatorname{Iso}}

\theoremstyle{definition}
\newtheorem{defn}{Definition}[section]
\newtheorem{exmpl}[defn]{Example}

\theoremstyle{plain}
\newtheorem{thm}[defn]{Theorem}
\newtheorem{lem}[defn]{Lemma}
\newtheorem{cor}[defn]{Corollary}

\begin{document}

\title{Orthogonal Eisenstein Series and Theta Lifts}
\author{Paul Kiefer}
\thanks{The author was supported by the LOEWE research unit USAG}
\thanks{Funded by the Deutsche Forschungsgemeinschaft (DFG, German Research Foundation) TRR 326 \textit{Geometry and Arithmetic of Uniformized Structures}, project number 444845124.}

\begin{abstract}
We show that the additive Borcherds lift of vector-valued non-holo\-mor\-phic Eisenstein series are orthogonal non-holomorphic Eisenstein series for $O(2, l)$. Using this we give another proof that they have a meromorphic continuation, calculate their Fourier expansion and show that they have a functional equation analogous to the classical case. Moreover, we will investigate the image of Borcherds lift and give a sufficient condition for surjectivity.
\end{abstract}

\maketitle

\tableofcontents

\section{Introduction}

Let $L$ be an even lattice of signature $(2, l)$ with $l \geq 4$ even. For an easier exposition we assume that $L$ is maximal but the results also hold for general lattices $L$ (in fact we will work with an arbitrary even lattice of signature $(b^+, b^-)$ and then specialize to signature $(2, l)$). Let $\rho_L$ be the Weil representation of the metaplectic group $\Mp_2(\IZ)$ on the group algebra $\IC[L' / L]$, see \cite{Borcherds}, \cite{BruinierChern}. Similarly as in \cite{BruinierKuehn}, \cite{BruinierKuss} for $\Re(s) \gg 0$ and $\tau = u + i v \in \IH$ we define the vector-valued non-holomorphic Eisenstein series of weight $k \in \IZ$ by
$$E_{k, 0}(\tau, s) := \frac{1}{2} \sum_{M \in \tilde{\Gamma}_\infty \bs \Mp_2(\IZ)} v^s \frake_0 \vert_{k, L} M,$$
where $\vert_{k, L}$ is the Petersson slash operator for the Weil representation $\rho_L$ (in \cite{BruinierKuehn}, \cite{BruinierKuss} the dual Weil representation $\rho_L^*$ is used) and $\tilde{\Gamma}_\infty \subseteq \Mp_2(\IZ)$ is an index $2$ subgroup of the stabilizer of $\infty$. According to \cite{BruinierKuehn} the Eisenstein series has a Fourier expansion of the form
$$2 \frake_0 v^s + c_{k, 0}(0, 0, s) v^{1 - s - k} + \sum_{\gamma \in L' / L} \sum_{\substack{n \in \IZ + q(\gamma) \\ n \neq 0}} c_{k, 0}(\gamma, n, s) \calW_s(4 \pi n v) \frake_\gamma(nu),$$
where $\calW_s$ is a Whittaker function and $\frake_\gamma(nu) := \frake_\gamma e^{2 \pi i u}$. The Fourier expansion yields a meromorphic continuation to all $s \in \IC$. Moreover, since $L$ is maximal one has the functional equation
$$E_{k,0}(\tau, s) = \frac{1}{2} c_{k, 0}(0, 0, s) E_{k,0}(\tau, 1 - k - s).$$
In general, this functional equation involves the vector-valued non-holomorphic Eisenstein series corresponding to $\beta \in L' / L$ with $q(\beta) = 0$. For a primitive isotropic $z \in L$ take $z' \in L'$ with $(z, z') = 1$. Define $K = L \cap z^\perp \cap z'^\perp$ and let $\IH_l = K \otimes \IR + i C$, where $C$ is a fixed connected component of
$$\{Y \in K \otimes \IR \mid q(Y) > 0 \}.$$
Then $\IH_l$ is a hermitian symmetric domain and there is an index $2$ subgroup $O^+(L \otimes \IR) \subseteq O(L \otimes \IR)$ acting on $\IH_l$. Write $L'$ for the dual lattice of $L$ and denote by $\Gamma(L)$ the kernel of the natural map $O^+(L) \to O(L' / L)$. For $\Re(s) \gg 0$ and $Z = X + i Y \in \IH_l$ one can define the orthogonal non-holomorphic Eisenstein series of weight $\kappa = \frac{l}{2} - 1 + k$ corresponding to the $0$-dimensional cusp $z$ by
$$\calE_{\kappa, z}(Z, s) := \sum_{\sigma \in \Gamma(L)_z \bs \Gamma(L)} q(Y)^s \vert_\kappa \sigma,$$
where $\vert_\kappa$ is the Petersson slash operator for $O^+(V)$ and $\Gamma(L)_z$ is the stabilizer of $z$. As in the vector-valued case we want to obtain a meromorphic continuation to all $s \in \IC$ and a functional equation. We will do this by writing the orthogonal non-holomorphic Eisenstein series as a Borcherds lift of vector-valued non-holomorphic Eisenstein series. Therefore, using the theta function
$$\Theta_L(\tau, Z)
= \frac{v^{\frac{l}{2}}}{2 (-2i)^{\kappa}} \sum_{\lambda \in L'} \frac{(\lambda, Z_L)^\kappa}{q(Y)^\kappa} \frake_\lambda(\tau q(\lambda_{Z_L}) + \overline{\tau} q(\lambda_{Z_L^\perp}))$$
we define the Borcherds lift as a regularized integral
$$\Phi_{k,0}(Z, s) := \int_{\SL_2(\IZ) \bs \IH}^{\text{reg}} \langle E_{k, 0}(\tau, s), \Theta_L(\tau, Z) \rangle v^k \mathrm{d}\mu,$$
see \cite{Borcherds}, \cite{BruinierChern}. We will first show that this definition makes sense for $\Re(s) \gg 0$ and that it has a meromorphic continuation to all of $s \in \IC$. In particular, we obtain a functional equation for $\Phi_{k,0}$ coming from the functional equation of the vector-valued non-holomorphic Eisenstein series $E_{k, 0}$. By unfolding against the vector-valued non-holomorphic Eisenstein series one obtains

\begin{thm}[{see Theorem \ref{thm:ThetaLiftIsEisensteinSeries}}]
We have
$$\Phi_{k,0}(Z, s) = 2 \frac{\Gamma(s + \kappa) \zeta(2s + \kappa)}{(-2 \pi i)^\kappa \pi^s} \calE_{\kappa, z}(Z, s).$$
\end{thm}

This shows, in particular, that the orthogonal non-holomorphic Eisenstein series have a meromorphic continuation to all $s \in \IC$. Using the functional equation of vector-valued non-holomorphic Eisenstein series we obtain the functional equation for orthogonal non-holomorphic Eisenstein series

\begin{cor}[{see Corollary \ref{cor:EisensteinSeriesFunctionalEquation}}]
The non-holomorphic Eisenstein series $\calE_{\kappa, z}(Z, s)$ has a functional equation relating the values at $s$ and $1 - k - s$. More precisely, if $L$ is maximal and if we write
$$\tilde{\varphi}_{k,0}(z, s) = \frac{\pi^s \Gamma(1 - k - s + \kappa) \zeta(2(1 - k - s) + \kappa)}{\pi^{1 - k - s}\Gamma(s + \kappa) \zeta(2s + \kappa)} c_{k,0}(0, 0, s),$$
then we have
$$\calE_{\kappa, z}(Z, s) = \frac{1}{2} \tilde{\varphi}_{k,0}(z, s) \calE_{\kappa, z}(Z, 1 - k - s).$$
\end{cor}

By unfolding against the theta function one can calculate the Fourier expansion using the method of \cite{Borcherds}.

\begin{thm}[{see Theorem \ref{thm:ThetaLiftFourierExpansion}}]
The non-holomorphic Eisenstein series has a Fourier expansion of the form
$$\calE_{\kappa, z} = \Phi^K_{0, \beta}(Y / \lvert Y \rvert, s) + \sum_{\lambda \in K'} b_{\kappa,z}(\lambda, Y, s) e(\lambda, X),$$
where $\Phi^K_{0, \beta}$ is a theta lift of some Eisenstein series corresponding to the sublattice $K$. For $\lambda = 0$ we have
$$b_{\kappa, z}(0, Y, s) = 2q(Y)^s + \tilde{\varphi}_{k,0}(z, s) q(Y)^{1 - s - k}.$$
For the other coefficients see Theorem \ref{thm:ThetaLiftFourierExpansion}.
\end{thm}

If $L$ is not maximal, we have

\begin{thm}[{see Theorem \ref{thm:ThetaLiftNonHolomorphicSurjective}}]
Assume that $L$ splits two hyperbolic planes. Then every orthogonal non-holomorphic Eisenstein series corresponding to a $0$-dimensional cusp for $\Gamma(L)$ is a theta lift of some vector-valued non-holomorphic Eisenstein series for the Weil representation $\rho_L$. In particular, they have a meromorphic continuation to all $s \in \IC$ and have a functional equation similar to the vector-valued case.
\end{thm}

More generally, Theorem \ref{thm:ThetaLiftNonHolomorphicSurjective} determines the image of the additive Borcherds lift for arbitrary lattices $L$.

\subsection*{Acknowledgment}

I would like to thank my advisor J.H. Bruinier for suggesting this topic as part of my doctoral thesis. Moreover, I would like to thank him for his support and our helpful discussions.

\section{Vector-Valued Non-Holomorphic Eisenstein Series}

We will now introduce the Weil representation and vector-valued modular forms. Therefore, let $\IH := \{\tau = u + iv \in \IC \mid v > 0 \}$ be the usual upper half-plane. For $z \in \IC$ we write $e(z) := e^{2 \pi i z}$ and we denote by $\sqrt{z} = z^{\frac{1}{2}}$ the principal branch of the square-root, i.e. $\arg(\sqrt{z}) \in (- \frac{\pi}{2}, \frac{\pi}{2}]$. 

\begin{defn}
We denote by $\Mp_2(\IR)$ the \emph{metaplectic cover} of $\SL_2(\IR)$. It is realized as pairs $(M, \phi)$, where $M \in \SL_2(\IR)$ and $\phi : \IH \to \IC$ is a holomorphic square root of $\tau \mapsto c \tau + d$. The product for $(M_1, \phi_1), (M_2, \phi_2) \in \Mp_2(\IR)$ is given by
$$(M_1, \phi_1(\tau))(M_2, \phi_2(\tau)) := (M_1 M_2, \phi_1(M_2 \tau) \phi_2(\tau)),$$
where $\left(\begin{smallmatrix} a & b \\ c & d\end{smallmatrix}\right) \tau = \frac{a \tau + d}{c \tau + d}$ is the usual action of $\SL_2(\IR)$.
\end{defn}
By $\Mp_2(\IZ)$ we denote the inverse image of $\SL_2(\IZ)$ under the covering map. It is generated by
$$T = \left(\begin{pmatrix}1 & 1 \\ 0 & 1\end{pmatrix}, 1\right) \quad \text{and} \quad S = \left(\begin{pmatrix}0 & -1 \\ 1 & 0\end{pmatrix}, \sqrt{\tau}\right).$$
We have the relation $S^2 = (ST)^3 = Z$, where $Z = \left(\left(\begin{smallmatrix}-1 & 0 \\ 0 & -1\end{smallmatrix}\right), i\right)$ is the standard generator of the center of $\Mp_2(\IZ)$. Furthermore, we will write $\Gamma_\infty := \{\left(\begin{smallmatrix}1 & n \\ 0 & 1\end{smallmatrix}\right) \mid n \in \IZ \}$ and $\tilde{\Gamma}_\infty := \{\left(\left(\begin{smallmatrix}1 & n \\ 0 & 1\end{smallmatrix}\right), 1\right) \mid n \in \IZ \}$. For an even non-degenerate lattice $L$ of signature $(b^+, b^-)$ consider the group ring $\IC[L' / L]$ with standard basis $(\frake_\gamma)_{\gamma \in L' / L}$. For $\frakv = \sum_{\gamma \in L' / L} \frakv_\gamma \frake_\gamma \in \IC[L' / L]$ we write $\frakv^* := \sum_{\gamma \in L' / L} \frakv_\gamma \frake_{-\gamma}$. Moreover, we write $\langle \cdot, \cdot \rangle$ for the standard inner product on $\IC[L' / L]$ which is anti-linear in the second variable. Write $\Iso(L' / L) \subseteq L' / L$ for the set of isotropic elements and denote the subspace of vectors that are supported on isotropic elements by $\Iso(\IC[L' / L])$. Moreover, we introduce the notation $\frake_\gamma(\tau) := e(\tau) \frake_\gamma = e^{2 \pi i \tau} \frake_\gamma$.

\begin{defn}[{\cite{Borcherds}}]
The \emph{Weil representation} is the unitary representation $\rho_L$ of $\Mp_2(\IZ)$ on $\IC[L' / L]$ defined by
$$\rho_L(T) \frake_\gamma := e(q(\gamma)) \frake_\gamma \quad \text{and} \quad \rho_L(S) \frake_\gamma := \frac{\sqrt{i}^{b^- - b^+}}{\sqrt{L' / L}} \sum_{\delta \in L' / L} e(-(\gamma, \delta)) \frake_\delta.$$
The Weil representation factors through a finite quotient of $\Mp_2(\IZ)$.
\end{defn}

A short calculation using orthogonality of characters shows for example 
$$\rho_L(Z) \frake_\gamma = i^{b^- - b^+} \frake_{-\gamma}.$$
For a vector-valued function $f : \IH \to \IC[L' / L]$ we write $f_\gamma : \IH \to \IC$ for its components with respect to the standard basis, i.e. $f = \sum_{\gamma \in L' / L} f_\gamma \frake_\gamma$. For $k \in \frac{1}{2} \IZ$ we define the \emph{Petersson slash operator} $f \mapsto f\vert_{k, L}(M, \phi)$ by
$$(f \vert_{k, L}(M, \phi))(\tau) = \phi(\tau)^{-2 k} \rho_L^{-1}(M, \phi) f(M \tau).$$
If $f : \IH \to \IC[L' / L]$ is smooth and invariant under the action of $T$, i.e. $f \vert_{k, L} T = f$, then we have a Fourier expansion
$$f(\tau) = \sum_{\gamma \in L' / L} \sum_{n \in \IZ + q(\gamma)} c(\gamma, n, v) \frake_\gamma(nu).$$

\begin{defn}
A function $f : \IH \to \IC[L' / L]$ is said to be \emph{modular} of weight $k$ with respect to the Weil representation $\rho_L$ if $f \vert_{k, L} (M, \phi) = f$ for all $(M, \phi) \in \Mp_2(\IZ)$.
\end{defn}

Assume now that $b^+ - b^-$ is even and let $k \in \IZ$. Moreover, set $\kappa = \frac{b^- - b^+}{2} + k$. Write $\Iso(L' / L)$ for the isotropic elements in $L' / L$ and let $\beta \in \Iso(L' / L)$. Similar as in \cite{BruinierKuehn} we define the \emph{vector-valued non-holomorphic Eisenstein series} of weight $k$ by
$$E_{k, \beta}(\tau, s) = \frac{1}{2} \sum_{M \in \tilde{\Gamma}_\infty \bs \Mp_2(\IZ)} v^s \frake_\beta \vert_{k, L} M.$$
Observe that \cite{BruinierKuehn} consider the dual Weil representation $\rho_L^*$. For $\beta \in \Iso(L' / L)$ of order $N_\beta$ and a character $\chi : (\IZ / N_\beta \IZ)^\times \to \IC^\times$ we define
$$E_{k, \beta, \chi}(\tau, s) := \sum_{n \in (\IZ / N_\beta \IZ)^\times} \chi(n) E_{k, n \beta}(\tau, s).$$
More generally, for $\frakv \in \Iso(\IC[L' / L])$ we define
\begin{align*}
E_{k, \frakv}(\tau, s)
&= \frac{1}{2} \sum_{M \in \tilde{\Gamma}_\infty \bs \Mp_2(\IZ)} v^s \frakv \vert_{k, L} M \\
&= \sum_{\beta \in \Iso(L' / L)} \frakv_\beta E_{k, \beta}(\tau, s).
\end{align*}
We have $E_{k, \frakv^*} = (-1)^\kappa E_{k, \frakv}$ and a Fourier expansion of the form
\begin{align*}
E_{k, \frakv}(\tau, s) &= (\frakv + (-1)^\kappa \frakv^*) v^s + \sum_{\gamma \in \Iso(L' / L)} c_{k, \frakv}(\gamma, 0, s) v^{1 - s - k} \\
&+ \sum_{\gamma \in L' / L} \sum_{\substack{n \in \IZ + q(\gamma) \\ n \neq 0}} c_{k, \frakv}(\gamma, n, s) \calW_s(4 \pi n v) \frake_\gamma(n u),
\end{align*}
where $\calW_s$ is the Whittaker function defined in \cite{BruinierKuehn}. For the precise coefficients see \cite{BruinierKuehn}, \cite{Williams} (or \cite{BruinierKuss}, \cite{Scheithauer}, \cite{Schwagenscheidt} for the holomorphic case $k > 2$), we will not need them here. We will also use the notation
\begin{align*}
c_{k, \frakv}(\gamma, n, s, v) &:= c_{k, \frakv}(\gamma, n, s) \calW_s(4 \pi n v), \\
c_{k, \frakv}(\gamma, 0, s, v) &:= (\frakv_\gamma + (-1)^\kappa \frakv_{-\gamma}) v^s + c_{k, \frakv}(\gamma, 0, s) v^{1 - s - k}.
\end{align*}
If $\kappa$ is odd, the Eisenstein series for $\beta \in \Iso(L' / L)$ with $\beta = -\beta$ vanish identically. It can be easily seen that the vector-valued Eisenstein series are $\Mp_2(\IZ)$-translates of usual scalar-valued Eisenstein series. In particular, using \cite[Section 4.10]{DiamondShurman} we obtain the meromorphic continuation of the non-holomorphic Eisenstein series. According to \cite[Page 372]{Hejhal} (his Eisenstein series are given by $y^{\frac{k}{2}} E_{k, \beta}(\tau, s - \frac{k}{2})$ and he considers more general representations) we have the functional equation
$$E_{k, \beta}(\tau, s) = \frac{1}{2} \sum_{\alpha \in \Iso(L' / L)} c_{k, \beta}(\alpha, 0, s) E_{k, \alpha}(\tau, 1 - k - s).$$

\section{Siegel Theta Function}

Let $p$ be a polynomial on $\IR^{(b^+, b^-)}$ which is homogeneous of degree $\kappa$ in the positive definite variables and independent of the negative definite variables. For an isometry $\nu : L \otimes \IR \to \IR^{(b^+, b^-)}$ we write $\nu^+$ and $\nu^-$ for the inverse image of $\IR^{(b^+, 0)}$ and $\IR^{(0, b^-)}$. For an element $\lambda \in L \otimes \IR$ we write $\lambda_{\nu^\pm}$ for the projection of $\lambda$ onto $\nu^\pm$. The positive definite \emph{majorant} associated to $\nu$ is then given by $q_\nu(\lambda) = q(\lambda_{\nu^+}) - q(\lambda_{\nu^-})$. For $\gamma \in L' / L, \tau \in \IH$ and an isometry $\nu : L \otimes \IR \to \IR^{(b^+, b^-)}$ following \cite{Borcherds} we define the \emph{Siegel theta function}
\begin{align*}
\theta_\gamma(\tau, \alpha, \beta, \nu, p)
&:= \sum_{\lambda \in \gamma + L} \exp\left(\frac{\Delta}{8 \pi v}\right)(p)(\nu(\lambda + \beta)) \\
&\times e(\tau q((\lambda + \beta)_{\nu^+}) + \overline{\tau} q((\lambda + \beta)_{\nu^-} - (\lambda + \beta / 2, \alpha))),
\end{align*}
where $\Delta$ is the usual \emph{Laplace operator} on $\IR^{b^+ + b^-}$ and $\alpha, \beta \in L \otimes \IR$. Moreover, we define
\begin{align*}
\Theta_L(\tau, \alpha, \beta, \nu, p)
&:= \sum_{\gamma \in L' / L} \theta_\gamma(\tau, \alpha, \beta, \nu, p) \frake_\gamma.
\end{align*}
For $\alpha = \beta = 0$ we write
\begin{align*}
\theta_\gamma(\tau, \nu, p) := \theta_\gamma(\tau, 0, 0, \nu, p)
&= \sum_{\lambda \in \gamma + L} \exp\left(-\frac{\Delta}{8 \pi v}\right)(p)(\nu(\lambda)) e(\tau q(\lambda_{\nu^+}) + \overline{\tau} q(\lambda_{\nu^-})) \\
&= \sum_{\lambda \in \gamma + L} \exp\left(-\frac{\Delta}{8 \pi v}\right)(p)(\nu(\lambda)) e(iv q_\nu(\lambda) + u q(\lambda))
\end{align*}
and
\begin{align*}
\Theta_L(\tau, \nu, p) := \Theta_L(\tau, 0, 0, \nu, p)
&= \sum_{\gamma \in L' / L} \theta_\gamma(\tau, \nu, p) \frake_\gamma.
\end{align*}
Using Poisson summation one obtains

\begin{thm}[{\cite[Theorem 4.1]{Borcherds}}]
For $(M, \phi) \in \Mp_2(\IZ), M = \left(\begin{smallmatrix}a & b \\ c & d\end{smallmatrix}\right)$ we have
$$\Theta_L(M\tau, a \alpha + b \beta, c \alpha + d \beta, \nu, p) = \phi(\tau)^{b^+ + 2 \kappa} \overline{\phi(\tau)}^{b^-} \rho_L(M, \phi) \Theta_L(\tau, \alpha, \beta, \nu, p).$$
In particular, for $\alpha = \beta = 0$, the theta function $\Theta_L(\tau, \nu, p)$ has weight $(\frac{b^+}{2} + \kappa, \frac{b^-}{2})$.
\end{thm}

We will need the following growth estimate which is proven as in the classical case.

\begin{lem}
The Siegel theta function satisfies
$$\theta_\gamma(\tau, \nu, p) = O(v^{-\frac{b^+}{2} - \kappa - \frac{b^-}{2}})$$
for $v \to 0$ uniformly in $u$.
\end{lem}

Borcherds shows in \cite{Borcherds} that the theta function can be written as a Poincar\'e series. We will indicate the construction, which is also given in \cite{BruinierChern}. Write $\Iso_0(L)$ for the primitive isotropic elements of $L$ and let $z \in \Iso_0(L), z' \in L'$ with $(z, z') = 1$. Let $N_z$ be the level of $z$ and define the lattice
$$K = L \cap z^\perp \cap z'^\perp.$$
Then $K$ has signature $(b^+ - 1, b^- - 1)$. For a vector $\lambda \in L \otimes \IR$ we write $\lambda_K$ for its orthogonal projection to $K \otimes \IR$, which is given by
$$\lambda_K = \lambda - (\lambda, z) z' + (\lambda, z)(z', z') z - (\lambda, z') z.$$
Let $\zeta \in L$ such that $(z, \zeta) = N_z$ and write
$$\zeta = \zeta_K + N_z z' + B z$$
for some $B \in \IQ$. Then we have
$$L = K \oplus \IZ \zeta + \IZ z.$$
Consider the sublattice
$$L_0' = \{ \lambda \in L' \vert (\lambda, z) \equiv 0 \bmod{N_z} \} \subseteq L'$$
and the projection
$$\pi : L_0' \to K', \lambda \mapsto \pi(\lambda) = \lambda_K + \frac{(\lambda, z)}{N_z} \zeta_K.$$
This projection induces a surjective map $L_0' / L \to K' / K$ which we also denote by $\pi$ and we have
$$L_0' / L = \{ \lambda \in L' / L \vert (\lambda, z) \equiv 0 \bmod{N_z} \}.$$
For an isometry $\nu : L \otimes \IR \to \IR^{(b^+, b^-)}$ we write $\omega^{\pm}$ for the orthogonal complement of $z_{\nu^\pm}$ in $\nu^\pm$. This yields a decomposition
$$L \otimes \IR = \omega^+ \oplus \IR z_{\nu^+} \oplus \omega^- \oplus \IR z_{\nu^-}$$
and for $\lambda \in L \otimes \IR$ we write $\lambda_{\omega^\pm}$ for the corresponding projections of $\lambda$ onto $\omega^\pm$. Additionally the map
$$\omega : L \otimes \IR \to \IR^{(b^+, b^-)}, \lambda \mapsto \nu(\lambda_{\omega^+} + \lambda_{\omega^-})$$
can be seen to be an isometry $K \otimes \IR \to \IR^{(b^+ - 1, b^- - 1)}$ by restriction. For a homogeneous polynomial $p$ as above we now define the homogeneous polynomials $p_{\omega, h}$ of degree $\kappa - h$ in the positive definite variables by
$$p(\nu(\lambda)) = \sum_{h} (\lambda, z_{\nu^+})^{h} p_{\omega, h}(\omega(\lambda)).$$
We have the following

\begin{thm}[{\cite[Theorem 5.2]{Borcherds}}]\label{thm:ThetaFunctionKExpansion}
Let $\mu = -z' + \frac{z_{\nu^+}}{2 z_{\nu^+}^2} + \frac{z_{\nu^-}}{2 z_{\nu^-}^2} \in L \otimes \IR$. Then
\begin{align*}
\theta_{\gamma + L}(\tau, \nu, p)
&= \frac{1}{\sqrt{2 v z_{\nu^+}^2}} \sum_{\substack{c, d \in \IZ \\ c \equiv (\gamma, z) \bmod N_z}} \sum_{h} (-2iv)^{-h} \\
&\times (c \overline{\tau} + d)^{h} e\left(-\frac{\lvert c \tau + d \rvert^2}{4 i v z_{\nu^+}^2} - (\gamma, z') d + q(z') cd \right) \\
&\times \theta_{K + \pi(\gamma - cz')}(\tau, d \mu_K, -c \mu_K, \omega, p_{\omega, h}).
\end{align*}
\end{thm}

\section{Regularized Theta Lift}

As before let $L$ be an even lattice of signature $(b^+, b^-)$ and $p$ a polynomial on $\IR^{(b^+, b^-)}$ which is homogeneous of degree $\kappa$ in the first $b^+$ variables and independent of the last $b^-$ variables (in fact the results of this and the next two sections can be generalized to include polynomials which are homogeneous also in the negative definite variables). Denote by $\Theta_L(\tau, \nu, p)$ the corresponding theta series. Let $k = \frac{b^+}{2} + \kappa - \frac{b^-}{2}$ and assume $b^+ - b^- \equiv 0 \bmod{2}$. For $\beta \in \Iso(L' / L)$ recall the vector-valued non-holomorphic Eisenstein series $E_{k, \beta}(\tau, s)$ of weight $k$ for the Weil representation $\rho_L$. Then
$$\langle E_{k, \beta}(\tau, s), \Theta_L(\tau, \nu, p) \rangle v^{\frac{b^+}{2} + \kappa}$$
is invariant under $\Mp_2(\IZ)$ in $\tau$. Hence, following \cite{Borcherds}, we define
\begin{align*}
\Phi_{k, \beta}(\nu, p, s, t) = \lim_{T \to \infty} \int_{\calF_T} \langle E_{k, \beta}(\tau, s), \Theta_L(\tau, \nu, p) \rangle v^{\frac{b^+}{2} + \kappa - t} \frac{\mathrm{d}u \mathrm{d}v}{v^2},
\end{align*}
where
$$\calF_T = \{ \tau = u + iv \in \calF \mid v \leq T \}$$
and set
$$\Phi_{k, \beta}(\nu, p, s) := \int_{\calF}^{\text{reg}} \langle E_{k, \beta}(\tau, s), \Theta_L(\tau, \nu, p) \rangle v^{\frac{b^+}{2} + \kappa} \frac{\mathrm{d}u \mathrm{d}v}{v^2} := \CT_{t = 0} \Phi_{k, \beta}(\nu, p, s, t),$$
where $\CT_{t = 0}$ denotes the constant term in the Laurent expansion at $t = 0$ (of the meromorphic continuation). For $\frakv \in \Iso(\IC[L' / L])$ we define $\Phi_{k, \frakv}$ analogously. By linearity of the theta lift, it often suffices to consider $\Phi_{k, \beta}$ for $\beta \in \Iso(L' / L)$ only.

\begin{thm}\label{thm:ThetaLiftPoles}
For an isometry $\nu : L \otimes \IR \to \IR^{b^+, b^-}$ the regularized theta integral converges and defines a holomorphic function for $\Re(t) > \Re(s) > 1$ which has a meromorphic continuation in $t$ to $t = 0$ and to all $s \in \IC$. The possible poles come from the poles of $E_{k, \frakv}(\tau, s)$ except for finitely many simple poles at $s = 1 - \frac{b^+}{2} + n, \frac{b^+}{2} - n, n \in \IN_{> 0}$, which do not occur if $p$ is harmonic of positive degree. Moreover, the functional equation of $E_{k, \frakv}(\tau, s)$ yields a functional equation for the lift $\Phi_{k, \frakv}(\nu, p, s)$, i.e. we have
$$\Phi_{k, \frakv}(\nu, p, s) = \frac{1}{2} \sum_{\alpha \in \Iso(L' / L)} c_{k, \frakv}(\alpha, 0, s) \Phi_{k, \frakv}(\nu, p, 1 - k - s).$$
\end{thm}

\begin{proof}
We have
\begin{align*}
\Phi_{k, \beta}(\nu, p, s, t)
&= \lim_{T \to \infty} \int_{\calF_T} \langle E_{k, \beta}(\tau, s), \Theta_L(\tau, \nu, p) \rangle v^{\frac{b^+}{2} + \kappa - t} \frac{\mathrm{d}u \mathrm{d}v}{v^2} \\
&= \int_{\calF_1} \langle E_{k, \beta}(\tau, s), \Theta_L(\tau, \nu, p) \rangle v^{\frac{b^+}{2} + \kappa - t} \frac{\mathrm{d}u \mathrm{d}v}{v^2} \\
&+ \int_{v = 1}^\infty \int_{u = 0}^1 \langle E_{k, \beta}(\tau, s), \Theta_L(\tau, \nu, p) \rangle v^{\frac{b^+}{2} + \kappa - t} \frac{\mathrm{d}u \mathrm{d}v}{v^2}.
\end{align*}
Since $\calF_1$ is compact and the integrand is holomorphic in $s$ and $t$, the first integral converges and defines a holomorphic function for all $s, t \in \IC$, where $E_{k, \beta}$ has no pole. Therefore it is sufficient to consider the second summand, i.e.
$$\varphi(\nu, p, s, t) = \int_{v = 1}^\infty \int_{u = 0}^1 \langle E_{k, \beta}(\tau, s), \Theta_L(\tau, \nu, p) \rangle v^{\frac{b^+}{2} + \kappa - t} \frac{\mathrm{d}u \mathrm{d}v}{v^2}.$$
Inserting the Fourier coefficients of $E_{k, \beta}$ and $\Theta_L$ yields
\begin{align*}
\varphi(\nu, p, s, t)
&= \int_{v = 1}^\infty \int_{u = 0}^1 \sum_{\lambda \in L'} \sum_{n \in \IZ + q(\lambda)} c_{k, \beta}(\lambda, n, s, v) \exp\left(-\frac{\Delta}{8 \pi v}\right)(\overline{p})(\nu(\lambda)) \\
&\times \exp(-2 \pi v q_\nu(\lambda)) e\left((n - q(\lambda)) u\right) v^{\frac{b^+}{2} + \kappa - t} \frac{\mathrm{d}u \mathrm{d}v}{v^2} \\
&= \int_{v = 1}^\infty \int_{u = 0}^1 \sum_{\lambda \in L'} \sum_{n \in \IZ} c_{k, \beta}(\lambda, n + q(\lambda), s, v) \exp\left(-\frac{\Delta}{8 \pi v}\right)(\overline{p})(\nu(\lambda)) \\
&\times \exp(-2 \pi v q_\nu(\lambda)) e\left(n u\right) v^{\frac{b^+}{2} + \kappa - t} \frac{\mathrm{d}u \mathrm{d}v}{v^2}.
\end{align*}
Carrying out the integration over $u$ yields
\begin{align*}
& \int_{v = 1}^\infty \sum_{\lambda \in L'} c_{k, \beta}(\lambda, q(\lambda), s, v) \exp\left(-\frac{\Delta}{8 \pi v}\right)(\overline{p})(\nu(\lambda)) \exp(-2 \pi v q_\nu(\lambda)) v^{\frac{b^+}{2} + \kappa - t} \frac{\mathrm{d}v}{v^2} \\
&= \int_{v = 1}^\infty c_{k, \beta}(0, 0, s, v) \exp\left(-\frac{\Delta}{8 \pi v}\right)(\overline{p})(0) v^{\frac{b^+}{2} + \kappa - t} \frac{\mathrm{d}v}{v^2} \\
&+ \int_{v = 1}^\infty \sum_{\substack{\lambda \in L' \setminus \{0\} \\ q(\lambda) = 0}} c_{k, \beta}(\lambda, 0, s, v) \exp\left(-\frac{\Delta}{8 \pi v}\right)(\overline{p})(\nu(\lambda)) \exp(-2 \pi v q_\nu(\lambda)) v^{\frac{b^+}{2} + \kappa - t} \frac{\mathrm{d}v}{v^2} \\
&+ \int_{v = 1}^\infty \sum_{\substack{\lambda \in L' \setminus \{0\} \\ q(\lambda) \neq 0}} c_{k, \beta}(\lambda, q(\lambda), s, v) \exp\left(-\frac{\Delta}{8 \pi v}\right)(\overline{p})(\nu(\lambda)) \exp(-2 \pi v q_\nu(\lambda)) v^{\frac{b^+}{2} + \kappa - t} \frac{\mathrm{d}v}{v^2}.
\end{align*}
Inserting
$$c_{k, \beta}(\lambda, 0, s, v) := (\delta_{\beta, \lambda} + (-1)^\kappa \delta_{-\beta,\lambda}) v^s + c_{k, \beta}(\lambda, 0, s) v^{1 - s - k}$$
yields for $\varphi(\nu, p, s, t)$
\begin{align*}
& 2 \delta_{\beta, 0} \int_{v = 1}^\infty \exp\left(-\frac{\Delta}{8 \pi v}\right)(\overline{p})(0) v^{\frac{b^+}{2} + \kappa + s - 2 - t} \mathrm{d}v \\
&+ c_{k, \beta}(0, 0, s) \int_{v = 1}^\infty \exp\left(-\frac{\Delta}{8 \pi v}\right)(\overline{p})(0) v^{\frac{b^+}{2} + \kappa - s - k - 1 - t} \mathrm{d}v \\
&+ \int_{v = 1}^\infty \sum_{\substack{\lambda \in \pm \beta + L \setminus \{0\} \\ q(\lambda) = 0}} \exp\left(-\frac{\Delta}{8 \pi v}\right)(\overline{p})(\nu(\lambda)) \exp(-2 \pi v q_\nu(\lambda)) v^{\frac{b^+}{2} + \kappa + s - 2 - t} \mathrm{d}v \\
&+ \int_{v = 1}^\infty \sum_{\substack{\lambda \in L' \setminus \{0\} \\ q(\lambda) = 0}} c_{k, \beta}(\lambda, 0, s) \exp\left(-\frac{\Delta}{8 \pi v}\right)(\overline{p})(\nu(\lambda)) \exp(-2 \pi v q_\nu(\lambda)) v^{\frac{b^+}{2} + \kappa - s - k - 1 - t} \mathrm{d}v \\
&+ \int_{v = 1}^\infty \sum_{\substack{\lambda \in L' \setminus \{0\} \\ q(\lambda) \neq 0}} c_{k, \beta}(\lambda, q(\lambda), s, v) \exp\left(-\frac{\Delta}{8 \pi v}\right)(\overline{p})(\nu(\lambda)) \exp(-2 \pi v q_\nu(\lambda)) v^{\frac{b^+}{2} + \kappa - 2 - t} \mathrm{d}v.
\end{align*}
The first integral converges for $\Re(s) < 1 - \frac{b^+}{2} + \Re(t)$ and is equal to
$$\sum_{n = 0}^\infty \frac{(\Delta^n \overline{p})(0)}{(-8 \pi)^n (t - s - \frac{b^+}{2} + 1 + n) n!}.$$
Similarly, the second integral converges for $\Re(s) > \frac{b^+}{2} - Re(t)$ and is equal to
$$\sum_{n = 0}^\infty \frac{(\Delta^n \overline{p})(0)}{(-8 \pi)^n (s - \frac{b^+}{2} + t + n) n!}.$$
Both are finite sums and define meromorphic functions in $s$ and $t$. The possible poles match the poles described in the theorem. The sum in the third integral is a subseries of a theta function attached to the positive definite majorant $q_\nu$ and thus the integral converges for all $s, t \in \IC$. The convergence of the last two integrals follows from the asymptotics
$$c_{k, \beta}(\gamma, n, s, v) = O(e^{-2 \pi \lvert n \rvert})$$
for $v \geq 1$ as $n \to \pm\infty$.
\end{proof}

\section{Unfolding Against $E_{k, \beta}(\tau, s)$}

We will now calculate the theta lift $\Phi_{k, \beta}(\nu, p, s)$ by unfolding against the Eisenstein series. Let $N \in \IN$ and $c \in \IZ / N \IZ$. Define the \emph{modified zeta functions}
$$\zeta_+^c(s) := \sum_{\substack{n \equiv c \bmod N \\ n > 0}} \frac{1}{n^s}, \quad \zeta^c(s) := \sum_{\substack{n \equiv c \bmod N \\ n \neq 0}} \frac{1}{n^s}.$$

\begin{thm}\label{thm:UnfoldingAgainstEisenstein}
The theta lift ist given by
$$\Phi_{k, \beta}(\nu, p, s) = 2 \sum_{\substack{\lambda \in \beta + L \\ q(\lambda) = 0 \\ \lambda \neq 0}} \sum_{n = 0}^\infty \frac{\Delta^n \overline{p}(\nu(\lambda))}{(-8 \pi)^n n!} \frac{\Gamma(s + \frac{b^+}{2} + \kappa - 1 - n)}{(2 \pi q_\nu(\lambda))^{s + \frac{b^+}{2} + \kappa - 1 - n}}.$$
The series converges for $\Re(s) \gg 0$. We can rewrite this to
\begin{align*}
\Phi_{k, \beta}(\nu, p, s)
&= 2 \sum_{\lambda \in \Iso_0(L')} \zeta_+^{k_{\lambda\beta}}(2s + b^+ + \kappa - 2) \sum_{j = 0}^\infty \frac{\Delta^j \overline{p}(\nu(\lambda))}{(-8\pi)^j j!} \frac{\Gamma(s + \frac{b^+}{2} + \kappa - 1 - j)}{(2 \pi q_\nu(\lambda))^{s + \frac{b^+}{2} + \kappa - 1 - j}},
\end{align*}
where for $\lambda \in \Iso_0(L')$ we let $k_{\lambda \beta} \in \IZ / N_\lambda \IZ$ with $k_{\lambda \beta} \lambda \in \beta + L$ and the summands with $\beta + L \cap \IZ \lambda = \emptyset$ are meant to be zero.
\end{thm}

\begin{proof}
Using the theta transformation formula we see that $\Phi_{k, \beta}(\nu, p, s, t)$ is given by
\begin{align*}
& \lim_{T \to \infty} \int_{\calF_T} \langle E_{k, \beta}(\tau, s), \Theta_L(\tau, \nu, p) \rangle v^{\frac{b^+}{2} + \kappa - t} \frac{\mathrm{d}u \mathrm{d}v}{v^2} \\
&= \frac{1}{2} \lim_{T \to \infty} \int_{\calF_T} \sum_{(M, \phi) \in \tilde{\Gamma}_\infty \bs \Mp_2(\IZ)} \langle (\frake_\beta v^s)\mid_{k, L}(M, \phi), \Theta_L(\tau, \nu, p) \rangle v^{\frac{b^+}{2} + \kappa - t} \frac{\mathrm{d}u \mathrm{d}v}{v^2} \\
&= \frac{1}{2} \lim_{T \to \infty} \int_{\calF_T} \sum_{(M, \phi) \in \tilde{\Gamma}_\infty \bs \Mp_2(\IZ)} \langle \phi(\tau)^{-k} \rho_L(M, \phi)^{-1} \frake_\beta \Im(M\tau)^s, \Theta_L(\tau, \nu, p) \rangle v^{\frac{b^+}{2} + \kappa - t} \frac{\mathrm{d}u \mathrm{d}v}{v^2} \\
&= \frac{1}{2} \lim_{T \to \infty} \int_{\calF_T} \sum_{(M, \phi) \in \tilde{\Gamma}_\infty \bs \Mp_2(\IZ)} \langle \frake_\beta \Im(M\tau)^s, \overline{\phi(\tau)}^{-k}\rho_L(M, \phi)\Theta_L(\tau, \nu, p) \rangle v^{\frac{b^+}{2} + \kappa - t} \frac{\mathrm{d}u \mathrm{d}v}{v^2} \\
&= \lim_{T \to \infty} \int_{\calF_T} \sum_{M \in \Gamma_\infty \bs \SL_2(\IZ)} \langle \frake_\beta, \Theta_L(M \tau, \nu, p) \rangle \Im(M \tau)^{\frac{b^+}{2} + \kappa + s} v^{-t}\frac{\mathrm{d}u \mathrm{d}v}{v^2} \\
&= \lim_{T \to \infty} \int_{\calF_T} \sum_{M \in \Gamma_\infty \bs \SL_2(\IZ)} \overline{\theta_\beta(M \tau, \nu, p)} \Im(M\tau)^{\frac{b^+}{2} + \kappa + s} v^{-t}\frac{\mathrm{d}u \mathrm{d}v}{v^2} \\
&= 2\lim_{T \to \infty} \int_{\calF_T} \overline{\theta_\beta(\tau, \nu, p)} v^{\frac{b^+}{2} + \kappa + s - t}\frac{\mathrm{d}u \mathrm{d}v}{v^2} \\
&+ \lim_{T \to \infty} \int_{\calF_T} \sum_{\substack{M = \left(\begin{smallmatrix}a & b \\ c & d\end{smallmatrix}\right) \in \Gamma_\infty \bs \SL_2(\IZ) \\ c \neq 0}} \overline{\theta_\beta(M \tau, \nu, p)} \Im(M\tau)^{\frac{b^+}{2} + \kappa + s} v^{-t}\frac{\mathrm{d}u \mathrm{d}v}{v^2}.
\end{align*}
By the growth of the theta function, the integral
$$\int_{\calG} \theta_\beta(\tau, \nu, p) v^{s + \frac{b^+}{2} + \kappa} \frac{\mathrm{d}u \mathrm{d}v}{v^2},$$
where
$$\calG = \{\tau = u + iv \in \IH \mid \lvert u \rvert \leq \frac{1}{2}, \lvert \tau \rvert \leq 1 \},$$
converges absolutely for $\Re(s) \gg 0$. Thus we can set $t = 0$, take the limit $T \to \infty$ and unfold the second summand to obtain for the constant term at $t = 0$
\begin{align*}
& \lim_{T \to \infty} \int_{\calF_T} \sum_{\substack{M = \left(\begin{smallmatrix}a & b \\ c & d\end{smallmatrix}\right) \in \Gamma_\infty \bs \SL_2(\IZ) \\ c \neq 0}} \overline{\theta_\beta(M \tau, \nu, p)} \Im(M\tau)^{\frac{b^+}{2} + \kappa + s} v^{-t}\frac{\mathrm{d}u \mathrm{d}v}{v^2} \\
&= \sum_{\substack{M = \left(\begin{smallmatrix}a & b \\ c & d\end{smallmatrix}\right) \in \Gamma_\infty \bs \SL_2(\IZ) \\ c \neq 0}} \int_{\calF} \overline{\theta_\beta(M \tau, \nu, p)} \Im(M\tau)^{\frac{b^+}{2} + \kappa + s} v^{-t}\frac{\mathrm{d}u \mathrm{d}v}{v^2} \\
&= 2 \int_{\calG} \overline{\theta_\beta(\tau, \nu, s)} v^{s + \frac{b^+}{2} + \kappa} \frac{\mathrm{d}u \mathrm{d}v}{v^2}.
\end{align*}
In the first summand we cut off $\calF_1$ and add it to the second integral to obtain
\begin{align*}
2 \CT_{t = 0} \int_{v = 1}^\infty \int_{u = 0}^1 \overline{\theta_\beta(\tau, \nu, s)} v^{s + \frac{b^+}{2} + \kappa - t} \frac{\mathrm{d}u \mathrm{d}v}{v^2} + 2 \int_{v = 0}^1 \int_{u = 0}^1 \overline{\theta_\beta(\tau, \nu, s)} v^{s + \frac{b^+}{2} + \kappa} \frac{\mathrm{d}u \mathrm{d}v}{v^2}.
\end{align*}
Now we insert the Fourier expansion
\begin{align*}
\theta_\beta(\tau, \nu, p) = \sum_{\lambda \in \beta + L} \exp\left(-\frac{\Delta}{8 \pi v}\right)(p)(\nu(\lambda)) \exp(-2 \pi v q_\nu(\lambda)) e(q(\lambda)u)
\end{align*}
which leads, by evaluating the integral over $u$ and using $q(\beta) \in \IZ$, to
\begin{align*}
& 2 \CT_{t = 0} \int_{v = 1}^\infty \int_{u = 0}^1 \sum_{\lambda \in \beta + L} \exp\left(-\frac{\Delta}{8 \pi v}\right)(\overline{p})(\nu(\lambda)) \\
&\exp(-2 \pi v q_\nu(\lambda)) e(-q(\lambda)u) v^{s + \frac{b^+}{2} + \kappa - t} \frac{\mathrm{d}u \mathrm{d}v}{v^2} \\
&+ 2 \int_{v = 0}^1 \int_{u = 0}^1 \sum_{\lambda \in \beta + L} \exp\left(-\frac{\Delta}{8 \pi v}\right)(\overline{p})(\nu(\lambda)) \\
&\exp(-2 \pi v q_\nu(\lambda)) e(-q(\lambda)u) v^{s + \frac{b^+}{2} + \kappa} \frac{\mathrm{d}u \mathrm{d}v}{v^2} \\
&= 2 \CT_{t = 0} \int_{v = 1}^\infty \sum_{\substack{\lambda \in \beta + L \\ q(\lambda) = 0}} \exp\left(-\frac{\Delta}{8 \pi v}\right)(\overline{p})(\nu(\lambda)) \exp(-2 \pi v q_\nu(\lambda)) v^{s + \frac{b^+}{2} + \kappa - t} \frac{\mathrm{d}v}{v^2} \\
&+ 2 \int_{v = 0}^1 \sum_{\substack{\lambda \in \beta + L \\ q(\lambda) = 0}} \exp\left(-\frac{\Delta}{8 \pi v}\right)(\overline{p})(\nu(\lambda)) \exp(-2 \pi v q_\nu(\lambda)) v^{s + \frac{b^+}{2} + \kappa} \frac{\mathrm{d}v}{v^2}.
\end{align*}
Observe that for the $\lambda = 0$ term we have
\begin{align*}
& 2 \CT_{t = 0} \int_{v = 1}^\infty \exp\left(-\frac{\Delta}{8 \pi v}\right)(\overline{p})(0) v^{s + \frac{b^+}{2} + \kappa - t} \frac{\mathrm{d}u \mathrm{d}v}{v^2} \\
&= -2 \int_{v = 0}^1 \exp\left(-\frac{\Delta}{8 \pi v}\right)(\overline{p})(0) v^{s + \frac{b^+}{2} + \kappa} \frac{\mathrm{d}v}{v^2}
\end{align*}
and the rest of the terms converge for $t = 0$. Hence we obtain
\begin{align*}
\Phi_{k, \beta}(\nu, p, s)
&= 2 \sum_{\substack{\lambda \in \beta + L \\ q(\lambda) = 0 \\ \lambda \neq 0}} \int_{v = 0}^\infty \exp\left(-\frac{\Delta}{8 \pi v}\right)(\overline{p})(\nu(\lambda)) \exp(-2 \pi v q_\nu(\lambda)) v^{s + \frac{b^+}{2} + \kappa} \frac{\mathrm{d}v}{v^2} \\
&= 2 \sum_{\substack{\lambda \in \beta + L \\ q(\lambda) = 0 \\ \lambda \neq 0}} \sum_{j = 0}^\infty \frac{\Delta^j \overline{p}(\nu(\lambda))}{(-8\pi)^j j!} \int_{v = 0}^\infty \exp(-2 \pi v q_\nu(\lambda)) v^{s + \frac{b^+}{2} + \kappa - j} \frac{\mathrm{d}v}{v^2} \\
&= 2 \sum_{\substack{\lambda \in \beta + L \\ q(\lambda) = 0 \\ \lambda \neq 0}} \sum_{j = 0}^\infty \frac{\Delta^j \overline{p}(\nu(\lambda))}{(-8\pi)^j j!} \frac{\Gamma(s + \frac{b^+}{2} + \kappa - 1 - j)}{(2 \pi q_\nu(\lambda))^{s + \frac{b^+}{2} + \kappa - 1 - j}}.
\end{align*}
For $\lambda \in \Iso_0(L')$ let $k_{\lambda \beta} \in \IZ / N_\lambda \IZ$ with $k_{\lambda \beta} \lambda \in \beta + L$. Then we have
\begin{align*}
\Phi_{k, \beta}(\nu, p, s)
&= 2 \sum_{\lambda \in \Iso_0(L')} \sum_{\substack{n = 1 \\ n \equiv k_{\lambda \beta}}}^\infty \sum_{j = 0}^\infty \frac{\Delta^j \overline{p}(\nu(n\lambda))}{(-8\pi)^j j!} \frac{\Gamma(s + \frac{b^+}{2} + \kappa - 1 - j)}{(2 \pi q_\nu(n\lambda))^{s + \frac{b^+}{2} + \kappa - 1 - j}} \\
&= 2 \sum_{\lambda \in \Iso_0(L')} \sum_{\substack{n = 1 \\ n \equiv k_{\lambda \beta}}}^\infty \sum_{j = 0}^\infty \frac{\Delta^j \overline{p}(\nu(\lambda))}{(-8\pi)^j j!} \frac{\Gamma(s + \frac{b^+}{2} + \kappa - 1 - j)}{(2 \pi q_\nu(\lambda))^{s + \frac{b^+}{2} + \kappa - 1 - j}}  \frac{1}{n^{2s + b^+ + \kappa - 2}} \\
&= 2 \sum_{\lambda \in \Iso_0(L')} \zeta_+^{k_{\lambda\beta}}(2s + b^+ + \kappa - 2) \sum_{j = 0}^\infty \frac{\Delta^j \overline{p}(\nu(\lambda))}{(-8\pi)^j j!} \frac{\Gamma(s + \frac{b^+}{2} + \kappa - 1 - j)}{(2 \pi q_\nu(\lambda))^{s + \frac{b^+}{2} + \kappa - 1 - j}},
\end{align*}
where the summands with $\beta + L \cap \IZ \lambda = \emptyset$ are meant to be zero.
\end{proof}

\section{Unfolding Against $\Theta_L(\tau, \nu, p)$}\label{sec:UnfoldingAgainstTheta}

Let $f : \IH \to \IC[L' / L]$ be a vector-valued modular function of weight $k$ with respect to the Weil representation. Define
$$f^K(\tau, r, t) := \sum_{\gamma \in K' / K} f^K_{\gamma + K}(\tau, r, t) \frake_\gamma,$$
where
$$f^K_{\gamma + K}(\tau, r, t) := \sum_{\substack{\lambda \in L_0' / L \\ \pi(\lambda) = \gamma}} e(-r(\lambda, z') + rtq(z')) f_{\lambda+tz'+L}(\tau).$$
Then
\begin{lem}[{\cite[Theorem 5.3]{Borcherds}}]
For $(M, \phi) \in \Mp_2(\IZ)$ with $\left(\begin{smallmatrix}a & b \\ c & d\end{smallmatrix}\right)$ we have
$$f^K(M\tau, ar+bt, cr+dt) = \phi(\tau)^{2k} \rho_K(M, \phi) F_K(\tau, r, t).$$
\end{lem}

Following Borcherds we have the following

\begin{thm}[{\cite[Theorem 7.1]{Borcherds}}]\label{thm:UnfoldingAgainstTheta}
The function $\Phi_{k, \beta}(\nu, p, s)$ is equal to the constant term at $t = 0$ of
\begin{align*}
&\frac{1}{\sqrt{2} \lvert z_{\nu^+} \rvert} \Phi^K_{k, \beta}(w, p_{w,0}, s) + \frac{\sqrt{2}}{\lvert z_{\nu^+} \rvert} \sum_{h} (2 i)^{-h} \\
&\times \sum_{\lambda \in K'} \sum_{\substack{\delta \in L_0' / L \\ \pi(\delta) = \lambda + K}} \sum_{j = 0}^{\infty} \frac{(-\Delta)^j(\overline{p}_{w,h})(w(\lambda))}{(8\pi)^j j!} \sum_{n = 1}^\infty n^{h} e(n((\lambda, \mu_K) + (\delta, z'))) \\
&\times \int_{v = 0}^\infty \exp\left(-\frac{\pi n^2}{2 v z_{\nu^+}^2} - 2 \pi v q_w(\lambda)\right) c_{k, \beta}(\delta, q(\lambda), s, v) v^{\frac{b^+ - 5}{2} + \kappa - h - j - t} \mathrm{d}v,
\end{align*}
where $\Phi^K_{k, \beta}(w, p_{w,0}, s)$ denotes the constant term at $t = 0$ of the regularized theta integral for $E^K_{k, \beta}(\tau, s, 0, 0)$, i.e.
$$\CT_{t = 0} \lim_{T \to \infty} \int_{\mathcal{F}_T} \langle E^K_{k, \beta}(\tau, s, 0, 0), \Theta_K(\tau, w, p_{w, 0}) \rangle v^{\frac{b^+ - 1}{2} + \kappa - t} \frac{\mathrm{d}x \mathrm{d}v}{v^2}.$$
\end{thm}

\begin{proof}
One uses the expansion of Theorem \ref{thm:ThetaFunctionKExpansion} and unfolds with respect to this Poincar\'e series. Then plug in the Fourier expansion of the Eisenstein series and the theta function with respect to $K$ and calculate the integral over $u$.
\end{proof}

\begin{lem}[{\cite[Lemma 7.3]{Borcherds}}]
For $\lambda = 0$ the integral of Theorem \ref{sec:UnfoldingAgainstTheta} is equal to
\begin{align*}
&(\delta_{\beta, \delta} + (-1)^\kappa \delta_{- \beta, \delta}) \left(\frac{\pi n^2}{2 z_{\nu^+}^2}\right)^{s + \frac{b^+ - 3}{2} + \kappa - h - j - t} \Gamma\left(-s - \frac{b^+ - 3}{2} - \kappa + h + j + t\right) \\
&+ c_{k, \beta}(\delta, 0, s) \left(\frac{\pi n^2}{2 z_{\nu^+}^2}\right)^{\frac{b^+ - 1}{2} + \kappa - h - j - k - s - t} \Gamma\left(-\frac{b^+ - 1}{2} - \kappa + h + j + k + s + t\right).
\end{align*}
Moreover, the term with $\lambda = 0$ is then given by
\begin{align*}
&\sum_{j = 0}^{\infty} \frac{(-\Delta)^j(\overline{p}_{w,h})(0)}{(8\pi)^j j!} \left(\frac{\pi}{2 z_{\nu^+}^2}\right)^{\frac{b^+ - 3}{2} + \kappa - h - j - t} \\
&\times \sum_{b, c \in \IZ / N_z \IZ} e\left(\frac{bc}{N_z}\right) \bigg( \Gamma\left(-s - \frac{b^+ - 3}{2} - \kappa + h + j + t\right) \left(\frac{\pi}{2 z_{\nu^+}^2}\right)^{s} \\
&\times (\delta_{\beta, \frac{bz}{N_z}} + (-1)^\kappa \delta_{- \beta, \frac{bz}{N_z}}) \zeta_+^c(-2s - b^+ + 3 - 2\kappa + h + 2j + 2t) \\
&+ \Gamma\left(-\frac{b^+ - 1}{2} - \kappa + h + j + k + s + t\right) \left(\frac{\pi}{2 z_{\nu^+}^2}\right)^{1 - k - s} \\
&\times c_{k, \beta}\left(\frac{bz}{N_z}, 0, s\right) \zeta_+^c(-b^+ + 1 - 2\kappa + h + 2j + 2k + 2s + 2t)\bigg).
\end{align*}
\end{lem}

\begin{proof}
Inserting the $c_{k, \beta}(\delta, q(\lambda), s, v) = (\delta_{\beta, \delta} + (-1)^\kappa \delta_{- \beta, \delta}) v^s + c_{k, \beta}(\delta, 0, s) v^{1 - k - s}$ yields
\begin{align*}
&(\delta_{\beta, \delta} + (-1)^\kappa \delta_{- \beta, \delta}) \int_{v = 0}^\infty \exp\left(-\frac{\pi n^2}{2 v z_{\nu^+}^2}\right) v^{s + \frac{b^+ - 5}{2} + \kappa - h - j - t} \mathrm{d}v \\
&+ c_{k, \beta}(\delta, 0, s) \int_{v = 0}^\infty \exp\left(-\frac{\pi n^2}{2 v z_{\nu^+}^2}\right) v^{\frac{b^+ - 3}{2} + \kappa - h - j - k - s - t} \mathrm{d}v.
\end{align*}
Using the $\Gamma$-integral
\begin{align*}
\int_{v = 0}^\infty \exp\left(- \frac{\alpha}{v}\right) v^{\beta} \mathrm{d}v = \alpha^{\beta + 1} \Gamma(-\beta - 1)
\end{align*}
for $\Re(\alpha) > 0$ and $\Re(\beta) < -1$ with $\alpha = \frac{\pi n^2}{2 z_{\nu^+}^2}$ we obtain
\begin{align*}
&(\delta_{\beta, \delta} + (-1)^\kappa \delta_{- \beta, \delta}) \left(\frac{\pi n^2}{2 z_{\nu^+}^2}\right)^{s + \frac{b^+ - 3}{2} + \kappa - h - j - t} \Gamma\left(-s - \frac{b^+ - 3}{2} - \kappa + h + j + t\right) \\
&+ c_{k, \beta}(\delta, 0, s) \left(\frac{\pi n^2}{2 z_{\nu^+}^2}\right)^{\frac{b^+ - 1}{2} + \kappa - h - j - k - s - t} \Gamma\left(-\frac{b^+ - 1}{2} - \kappa + h + j + k + s + t\right).
\end{align*}
This yields for the summand with $\lambda = 0$
\begin{align*}
&\sum_{\substack{\delta \in L_0' / L \\ \pi(\delta) = 0 + K}} \sum_{j = 0}^{\infty} \frac{(-\Delta)^j(\overline{p}_{w,h})(0)}{(8\pi)^j j!} \sum_{n = 1}^\infty n^{h} e(n(\delta, z')) \\
&\times \bigg((\delta_{\beta, \delta} + (-1)^\kappa \delta_{- \beta, \delta}) \left(\frac{\pi n^2}{2 z_{\nu^+}^2}\right)^{s + \frac{b^+ - 3}{2} + \kappa - h - j - t} \Gamma\left(-s - \frac{b^+ - 3}{2} - \kappa + h + j + t\right) \\
&+ c_{k, \beta}(\delta, 0, s) \left(\frac{\pi n^2}{2 z_{\nu^+}^2}\right)^{\frac{b^+ - 1}{2} + \kappa - h - j - k - s - t} \Gamma\left(-\frac{b^+ - 1}{2} - \kappa + h + j + k + s + t\right)\bigg).
\end{align*}
Using that a set of representatives for $\delta \in L_0' / L$ with $p(\delta) = 0 + K$ is given by $\frac{bz}{N_z}$ where $b$ runs through $\IZ / N_z \IZ$ yields
\begin{align*}
&\sum_{j = 0}^{\infty} \frac{(-\Delta)^j(\overline{p}_{w,h})(0)}{(8\pi)^j j!} \Gamma\left(-s - \frac{b^+ - 3}{2} - \kappa + h + j + t\right) \\
&\times \left(\frac{\pi}{2 z_{\nu^+}^2}\right)^{s + \frac{b^+ - 3}{2} + \kappa - h - j - t} \sum_{b \in \IZ / N_z \IZ} (\delta_{\beta, \frac{bz}{N_z}} + (-1)^\kappa \delta_{- \beta, \frac{bz}{N_z}}) \\
&\times \sum_{n = 1}^\infty e\left(\frac{nb}{N_z}\right) n^{2s + b^+ - 3 + 2\kappa - h - 2j - 2t} \\
&+\sum_{j = 0}^{\infty} \frac{(-\Delta)^j(\overline{p}_{w,h})(0)}{(8\pi)^j j!} \Gamma\left(-\frac{b^+ - 1}{2} - \kappa + h + j + k + s + t\right) \\
&\times \left(\frac{\pi}{2 z_{\nu^+}^2}\right)^{\frac{b^+ - 1}{2} + \kappa - h - j - k - s - t} \sum_{b \in \IZ / N_z \IZ} c_{k, \beta}\left(\frac{bz}{N_z}, 0, s\right) \\
&\times \sum_{n = 1}^\infty e\left(\frac{nb}{N_z}\right) n^{b^+ - 1 + 2\kappa - h - 2j - 2k - 2s - 2t}.
\end{align*}
This shows the result.
\end{proof}

\begin{lem}[{\cite[Lemma 7.2]{Borcherds}}]
For $\lambda \neq 0$ with $q(\lambda) = 0$ the integral of Theorem \ref{sec:UnfoldingAgainstTheta} is equal to
\begin{align*}
&2 (\delta_{\beta, \delta} + (-1)^\kappa \delta_{-\beta, \delta}) \left(\frac{n}{2 \lvert \lambda_{\nu^+} \rvert \lvert z_{\nu^+} \rvert}\right)^{s + \frac{b^+ - 3}{2} + \kappa - h - j - t} K_{s + \frac{b^+ - 3}{2} + \kappa - h - j - t}\left(2\pi n \frac{\lvert \lambda_{\nu^+} \rvert}{\lvert z_{\nu^+} \rvert}\right) \\
&+ 2 c_\beta(\delta, 0, s) \left(\frac{n}{2 \lvert \lambda_{\nu^+} \rvert \lvert z_{\nu^+} \rvert}\right)^{\frac{b^+ - 1}{2} + \kappa - h - j - k - s - t} K_{\frac{b^+ - 1}{2} + \kappa - h - j - k - s - t}\left(2\pi n \frac{\lvert \lambda_{\nu^+} \rvert}{\lvert z_{\nu^+} \rvert}\right).
\end{align*}
\end{lem}

\begin{proof}
Again we insert the Fourier coefficients to obtain
\begin{align*}
&(\delta_{\beta, \delta} + (-1)^\kappa \delta_{-\beta, \delta}) \int_{v = 0}^\infty \exp\left(-\frac{\pi n^2}{2 v z_{\nu^+}^2} - 2 \pi v q_w(\lambda)\right) v^{s + \frac{b^+ - 5}{2} + \kappa - h - j - t} \mathrm{d}v \\
&+ c_{k, \beta}(\delta, 0, s) \int_{v = 0}^\infty \exp\left(-\frac{\pi n^2}{2 v z_{\nu^+}^2} - 2 \pi v q_w(\lambda)\right) v^{\frac{b^+ - 3}{2} + \kappa - h - j - k - s - t} \mathrm{d}v.
\end{align*}
We use the formula \cite[p. 313, 6.3(17)]{Erdelyi}
\begin{align*}
\int_{v = 0}^\infty \exp\left(- \alpha v - \frac{\beta}{v}\right) v^{\gamma} \mathrm{d}v = 2 \left(\frac{\beta}{\alpha}\right)^{\frac{1}{2} (\gamma + 1)} K_{\gamma + 1}(2\sqrt{\alpha \beta})
\end{align*}
for $\Re(\alpha), \Re(\beta) > 0$ with $\alpha = 2 \pi q_w(\lambda), \beta = \frac{\pi n^2}{2 z_{\nu^+}^2}$. This yields
\begin{align*}
&2 (\delta_{\beta, \delta} + (-1)^\kappa \delta_{-\beta, \delta}) \left(\frac{n^2}{4 q_w(\lambda) z_{\nu^+}^2}\right)^{\frac{1}{2} (s + \frac{b^+ - 3}{2} + \kappa - h - j - t)} \\
&\times K_{s + \frac{b^+ - 3}{2} + \kappa - h - j - t}\left(2\pi n \sqrt{\frac{q_w(\lambda)}{z_{\nu^+}^2}}\right) \\
&+ 2 c_{k, \beta}(\delta, 0, s) \left(\frac{n^2}{4 q_w(\lambda) z_{\nu^+}^2}\right)^{\frac{1}{2} (\frac{b^+ - 1}{2} + \kappa - h - j - k - s - t)} \\
&\times K_{\frac{b^+ - 1}{2} + \kappa - h - j - k - s - t}\left(2\pi n \sqrt{\frac{q_w(\lambda)}{z_{\nu^+}^2}}\right).
\end{align*}
Using $q_w(\lambda) = \lambda_{\omega^+}^2$ if $q(\lambda) = 0$ we can rewrite this to
\begin{align*}
&2 (\delta_{\beta, \delta} + (-1)^\kappa \delta_{-\beta, \delta}) \left(\frac{n}{2 \lvert \lambda_{\omega^+} \rvert \lvert z_{\nu^+} \rvert}\right)^{s + \frac{b^+ - 3}{2} + \kappa - h - j - t} K_{s + \frac{b^+ - 3}{2} + \kappa - h - j - t}\left(2\pi n \frac{\lvert \lambda_{\omega^+} \rvert}{\lvert z_{\nu^+} \rvert}\right) \\
&+ 2 c_{k, \beta}(\delta, 0, s) \left(\frac{n}{2 \lvert \lambda_{\omega^+} \rvert \lvert z_{\nu^+} \rvert}\right)^{\frac{b^+ - 1}{2} + \kappa - h - j - k - s - t} K_{\frac{b^+ - 1}{2} + \kappa - h - j - k - s - t}\left(2\pi n \frac{\lvert \lambda_{\omega^+} \rvert}{\lvert z_{\nu^+} \rvert}\right).
\end{align*}
\end{proof}

\section{Orthogonal Non-Holomorphic Eisenstein Series}

From now on let $L$ be an even lattice of signature $(2, l)$ and let $V = L \otimes \IQ, V(\IR) = V \otimes \IR, V(\IC) = V \otimes \IC$. Write $P(V(\IC))$ for the corresponding \emph{projective space}. For elements $Z_L = X_L + i Y_L \in V(\IC) \setminus \{0\}$ we write $[Z_L]$ for the canonical projection to $P(V(\IC))$. The subset
$$\calK = \{[Z_L] \in P(V(\IC)) \mid (Z_L, Z_L) = 0, (Z_L, \overline{Z_L}) > 0 \}$$
is the hermitian symmetric domain associated to $O(V)$. It has two connected components which are interchanged by $[Z_L] \mapsto [\overline{Z_L}]$. We choose one of them and call it $\calK^+$. The action of $O(V(\IR))$ on $V(\IR)$ induces an action on $\calK$. Let $O^+(V(\IR))$ be the subgroup which preserves the connected components of $\calK$ and let
$$\tilde{\calK}^+ = \{Z_L \in V(\IC) \setminus \{0\} \mid [Z_L] \in \calK^+ \}$$
be the preimage of $\calK^+$ under the projection. For $z \in \Iso_0(L)$ and $z' \in L'$ with $(z, z') = 1$ write $K = L \cap z^\perp \cap z'^\perp$. Let $d \in \Iso_0(K)$, $d' \in K'$ with $(d, d') = 1$ and $D = K \cap d^\perp \cap d'^\perp$. Moreover, let $\tilde{z} = z' - q(z') z$ and $\tilde{d} = d' - q(d') d$. Let $d_3, \ldots, d_l$ be a basis of $D$. Then $\tilde{d}, d, d_3, \ldots, d_l$ is a basis of $K \otimes \IR$. We define the orthogonal upper half plane as
$$\IH_l := \{Z = X + iY \in W(\IC) = K \otimes \IC \mid q(Y) > 0, (Y, d) > 0\}.$$
We will write $Z = z_1 \tilde{d} + z_2 d + Z_D$ with $Z_D \in D \otimes \IC$ and analogously for $X, Y \in W(\IR)$. For $Z \in \IH_l$ we define
$$Z_L := Z - q(Z) z + \tilde{z}.$$
If the component $\calK^+$ is chosen properly, this yields a biholomorphic map $Z \mapsto [Z_L]$. By setting $i C = \IH_l \cap i (K \otimes \IR)$ we see that $\IH_l = K \otimes \IR + iC$ is a \emph{tube domain}. The action of $O^+(V(\IR))$ on $\calK^+$ then induces an action of $O^+(V(\IR))$ on $\IH_l$. Let
$$j : O^+(V) \times \IH_l \to \IC^\times, \quad j(\sigma, Z) := (\sigma(Z_L), z)$$
be the \emph{factor of automorphy}, so that we have
$$j(\sigma, Z)(\sigma Z)_L = \sigma(Z_L)$$
and the cocycle relation
$$j(\sigma_1 \sigma_2, Z) = j(\sigma_1, \sigma_2 Z) j(\sigma_2, Z).$$
According to \cite[Lemma 3.20]{BruinierChern} we have
$$q(\Im(\sigma Z)) = \frac{q(\Im(Z))}{\lvert j(\sigma, Z) \rvert^2}.$$

\begin{defn}
A function $f : \tilde{\calK}^+ \to \IC$ is called modular of weight $\kappa \in \IZ$ with respect to $\Gamma \subseteq \Gamma(L)$ if it satisfies
\begin{enumerate}
\item[(i)] $f(t Z_L) = t^{-\kappa} f(Z_L)$ for all $t \in \IC^\times$.
\item[(ii)] $f(\sigma Z_L) = f(Z_L)$ for all $\sigma \in \Gamma$.
\end{enumerate}
\end{defn}

For a modular function $f : \tilde{\calK}^+ \to \IC$ of weight $\kappa \in \IZ$ with respect to $\Gamma$ define
$$f_z : \IH_l \to \IC, \quad f_z(Z) := f(Z_L) = f(Z - q(Z) z + \tilde{z}).$$
Then $f_z$ satisfies
$$f_z(\sigma Z) = j(\sigma, Z)^\kappa f_z(Z)$$
for all $\sigma \in \Gamma$ and we have a bijective correspondence between modular functions and functions with this transformation property. If we define the weight $\kappa$ slash operator $f_z \vert_\kappa \sigma$ for $\sigma \in O^+(V)$ by
$$(f_z \vert_\kappa \sigma)(Z) := j(\sigma, Z)^{-\kappa} f_z(\sigma Z),$$
then the modular functions on $\IH_l$ are exactly the functions that are invariant under the slash operator for all $\sigma \in \Gamma$.

Let now $\kappa \in \IZ$ and $\lambda \in \Iso_0(L)$ and recall the tube domain representation $\IH_l$ corresponding to a fixed cusp $z \in \Iso_0(L)$. Denote by $\Gamma(L)_\lambda \subseteq \Gamma(L)$ the stabilizer of $\lambda$ in $\Gamma(L)$ and write $\sigma_{\lambda} \in O^+(V)$ for an element satisfying $\sigma_{\lambda} \lambda = z$. Then $q(Y)^s \vert_\kappa \sigma_{\lambda}$ has weight $\kappa$ with respect to $\Gamma(L)_\lambda$. Hence we define the non-holomorphic Eisenstein series corresponding to the $0$-dimensional cusp $\lambda$ by
\begin{align*}
\calE_{\kappa, \lambda}(Z, s)
&:= \sum_{\sigma \in \Gamma(L)_\lambda \bs \Gamma(L)} q(Y)^s \vert_\kappa \sigma_{\lambda}\sigma \\
&= \sum_{\sigma \in \Gamma(L)_\lambda \bs \Gamma(L)} j(\sigma_{\lambda} \sigma, Z)^{-\kappa} \left(\frac{q(Y)}{\lvert j(\sigma_{\lambda} \sigma, Z) \rvert^2}\right)^s \\
&= \sum_{\sigma \in \Gamma(L)_\lambda \bs \Gamma(L)} (\lambda, \sigma(Z_L))^{-\kappa} \left(\frac{q(Y)}{\lvert (\lambda, \sigma(Z_L)) \rvert^2}\right)^s
\end{align*}
for $Z \in \IH_l$ and $\Re(s) \gg 0$. The Eisenstein series does not depend on the choice of $\sigma_\lambda$. We have $\Gamma(L)_{-\lambda} = \Gamma(L)_\lambda$ and $\calE_{\kappa, -\lambda}(Z, s) = (-1)^\kappa \calE_{\kappa, \lambda}(Z, s)$.

Consider the map
$$\pi_L : \Gamma(L) \bs \Iso_0(L) \to \Gamma(L) \bs \Iso_0(L') \to \Iso(L' / L),$$
where the first arrow is given by $z \mapsto \frac{z}{N_z}$ with $N_z$ being the level of $z$. For $\delta \in \Iso(L' / L)$ we let
$$\calG_{\kappa, \delta} := \sum_{\lambda \in \pi_L^{-1}(\delta)} \calE_{\kappa, \lambda},$$
in particular, $\calG_{\kappa, \delta} = 0$ if the preimage is empty and if $\delta = -\delta$ for odd $\kappa$. For $\delta \in \Iso(L' / L)$ of order $N_\delta$ and $\chi$ a Dirichlet character of modulus $N_\delta$ we define
$$\calG_{\kappa, \delta, \chi} := \sum_{k \in (\IZ / N_\delta \IZ)^\times} \chi(k) \calG_{\kappa, k \delta}.$$
By orthogonality of characters, the space generated by $\calG_{\kappa, \delta}$ and the space generated by $\calG_{\kappa, \delta, \chi}$ coincide. Moreover, one easily sees, that injectivity (resp. surjectivity) of $\pi_L$ implies that the map induced by $\delta \mapsto \calG_{\kappa, \delta}, \delta \in \Iso(L' / L)$ is surjective (resp. injective). The following lemmas investigate the injectivity and surjectivity of $\pi_L$.

\begin{lem}
The map $\pi_L$ is surjective if and only if $L$ splits a hyperbolic plane.
\end{lem}

\begin{proof}
If $\pi_L$ is surjective, then there is some primitive isotropic element in $L$ of level $1$, which yields a splitting of a hyperbolic plane. Conversely, assume
$$L = U \oplus L_1 = e_1 \IZ \oplus e_2 \IZ + L_1,$$
and let $\delta \in L' / L = L_1' / L_1$ be an isotropic element. Consider a preimage $\lambda \in L_1'$. Then $e_1 - q(\lambda) e_2 + \lambda \in L'$ is primitive isotropic with image $\delta \in L' / L$.
\end{proof}

\begin{lem}[{\cite[Lemma 4.4]{FreitagHermann}}]
If $L$ splits two hyperbolic planes, then $\pi_L$ is bijective.
\end{lem}

\section{Orthogonal Eisenstein Series as Theta Lifts}

Let now $L$ be an even lattice of signature $(2, l)$ with $l \equiv 0 \bmod{2}$. Let $\kappa = \frac{l}{2} - 1 + k$ and $p(x_1, x_2) = (x_1 + ix_2)^\kappa$. Recall the identification $\nu : \calK^+ \to \Gr(L), [Z_L] = [X_L + iY_L] \mapsto \nu(Z_L) = \IR X_L + \IR Y_L$. For $Z_L = X_L + iY_L \in \tilde{\calK}^+$ the map
$$a X_L \to b Y_L \mapsto \lvert Y \rvert \begin{pmatrix}a \\ b\end{pmatrix}$$
defines an isometry $\nu_{Z_L} : \nu(Z_L) \to \IR^{(2, 0)}$ and by abuse of notation we will write $\nu_{Z_L}$ for every isometry which equals $\nu_{Z_L}$ on $\nu(Z_L)$. For $\lambda \in V(\IR)$ we write $\lambda_{Z_L}$ and $\lambda_{Z_L^\perp}$ for the projection of $\lambda$ to $\nu(Z_L)$ and $\nu(Z_L)^\perp$. Then $q(\lambda) = q(\lambda_{Z_L}) + q(\lambda_{Z_L^\perp})$ and we denote by $q_{Z_L}(\lambda) = q(\lambda_{Z_L}) - q(\lambda_{Z_L^\perp})$ the positive definite majorant. We have
\begin{align*}
z_{Z_L} &= \frac{1}{Y^2} X_L \\
\lvert z_{Z_L} \rvert &= \frac{1}{\lvert Y \rvert} \\
\mu_K &= X \\
\omega^+ &= \IR Y_L \\
\lambda_{\omega^+} &= \frac{(\lambda, Y_L)}{Y^2} Y_L \\
\lvert \lambda_{\omega^+} \rvert &= \lvert (\lambda, Y_L) \rvert \lvert z_{Z_L} \rvert.
\end{align*}
Now, the map
$$\lambda \mapsto p(\nu_{Z_L}(\lambda)) = \frac{(\lambda, Z_L)^\kappa}{\lvert Y \rvert^\kappa}$$
is well-defined, since $p$ only depends on the positive definite variables. Thus we define
\begin{align*}
\Theta_L(\tau, Z)
&= \frac{i^\kappa v^{\frac{l}{2}}}{2 \lvert Y \rvert^\kappa} \Theta_L(\tau, \nu_{Z_L}, p) \\
&= \frac{v^{\frac{l}{2}}}{2 (-2i)^{\kappa}} \sum_{\lambda \in L'} \frac{(\lambda, Z_L)^\kappa}{q(Y)^\kappa} \frake_\lambda(\tau q(\lambda_{Z_L}) + \overline{\tau} q(\lambda_{Z_L^\perp}))
\end{align*}
and write $\Phi_{k, \beta}(Z, s) := \frac{i^\kappa}{2 \lvert Y \rvert^\kappa} \Phi_{k, \beta}(\nu_{Z_L}, p, s)$. Then $\Phi_{k, \beta}(Z, s)$ is modular with respect to $\Gamma(L)$. Theorem \ref{thm:ThetaLiftPoles} yields a functional equation for the lift $\Phi_{k, \frakv}(Z, s)$, i.e. we have
$$\Phi_{k, \frakv}(Z, s) = \frac{1}{2} \sum_{\alpha \in \Iso(L' / L)} c_{k, \frakv}(\alpha, 0, s) \Phi_{k, \alpha}(\tau, 1 - k - s)$$
and we see that the poles of $\Phi_{k, \frakv}(Z, s)$ come from the poles of $E_{k, \frakv}(\tau, s)$ if $\kappa > 0$.

\begin{thm}\label{thm:ThetaLiftIsEisensteinSeries}
The theta lift of $E_{k, \beta}(\cdot, s)$ is equal to
\begin{align*}
\Phi_{k, \beta}(Z, s) &= \frac{\Gamma(s + \kappa)}{(-2\pi i)^{\kappa}\pi^{s}} \sum_{\lambda \in \Gamma(L) \bs \Iso_0(L)} N_\lambda^{2s + \kappa} \zeta_+^{k_{\lambda\beta}}(2s + \kappa) \calE_{\kappa, \lambda}(Z, s) \\
&= \frac{\Gamma(s + \kappa)}{(-2\pi i)^{\kappa}\pi^{s}} \sum_{\delta \in \Iso(L' / L)} N_\delta^{2s + \kappa} \zeta_+^{k_{\delta\beta}}(2s + \kappa) \calG_{\kappa, \delta}(Z, s),
\end{align*}
where $k_{\delta\beta} \in \IZ / N_\delta \IZ$ with $\beta = k_{\delta \beta} \delta$ (and the corresponding summands vanish if such a $k_{\delta \beta}$ does not exist), $N_\lambda$ is the level of $\lambda$ and $N_\delta$ is the order of $\delta$.
\end{thm}

\begin{proof}
We have $\Delta \overline{p} = 0$ and thus by Theorem \ref{thm:UnfoldingAgainstEisenstein}
$$\Phi_{k, \beta}(Z, s) = \frac{\Gamma(s + \kappa)}{(-2i)^{\kappa}} \sum_{\lambda \in \Iso_0(L')} \zeta_+^{k_{\lambda\beta}}(2s + \kappa) \frac{(\lambda, \overline{Z_L})^\kappa}{q(Y)^\kappa} \frac{1}{(2 \pi q_{Z_L}(\lambda))^{s + \kappa}},$$
where, again, the summands with $\beta + L \cap \IZ \lambda = \emptyset$ are meant to be zero. For $q(\lambda) = 0$ we have
$$2 q_{Z_L}(\lambda) = 4 q(\lambda_{Z_L}) = \frac{\lvert(\lambda, Z_L)\rvert^2}{q(Y)}$$
and thus
\begin{align*}
\Phi_{k, \beta}(Z, s)
&= \frac{\Gamma(s + \kappa)}{(-2\pi i)^{\kappa}\pi^{s}} \sum_{\lambda \in \Iso_0(L')} \frac{\zeta_+^{k_{\lambda\beta}}(2s + \kappa)}{(\lambda, Z_L)^\kappa} \left(\frac{q(Y)}{\lvert (\lambda, Z_L) \rvert^2}\right)^s \\
&= \frac{\Gamma(s + \kappa)}{(-2\pi i)^{\kappa}\pi^{s}} \sum_{\lambda \in \Gamma(L) \bs \Iso_0(L')} \sum_{\sigma \in \Gamma(L)_\lambda \bs \Gamma(L)} \frac{\zeta_+^{k_{\sigma(\lambda)\beta}}(2s + \kappa)}{(\sigma(\lambda), Z_L)^\kappa} \left(\frac{q(Y)}{\lvert (\sigma(\lambda), Z_L) \rvert^2}\right)^s.
\end{align*}
Since we have $k_{\sigma(\lambda)\beta} = k_{\lambda \beta}$ and $\sigma(\beta + L) = \beta + L$ for $\sigma \in \Gamma(L)$, this yields
\begin{align*}
& \frac{\Gamma(s + \kappa)}{(-2\pi i)^{\kappa}\pi^{s}} \sum_{\lambda \in \Gamma(L) \bs \Iso_0(L')} \zeta_+^{k_{\lambda\beta}}(2s + \kappa) \sum_{\sigma \in \Gamma(L)_\lambda \bs \Gamma(L)}  \frac{1}{(\sigma(\lambda), Z_L)^\kappa} \left(\frac{q(Y)}{\lvert (\sigma(\lambda), Z_L) \rvert^2}\right)^s \\
&= \frac{\Gamma(s + \kappa)}{(-2\pi i)^{\kappa}\pi^{s}} \sum_{\lambda \in \Gamma(L) \bs \Iso_0(L)} N_\lambda^{2s + \kappa} \zeta_+^{k_{\lambda\beta}}(2s + \kappa) \calG_{\kappa, \lambda}(Z, s) \\
&= \frac{\Gamma(s + \kappa)}{(-2\pi i)^{\kappa}\pi^{s}} \sum_{\delta \in \Iso(L' / L)} N_\delta^{2s + \kappa} \zeta_+^{k_{\delta\beta}}(2s + \kappa) \calG_{\kappa, \delta}(Z, s)
\end{align*}
\end{proof}

In particular, we obtain a map from $\Iso(\IC[L' / L])$ to the space of non-holomorphic Eisenstein series sending $\frakv \in \Iso(\IC[L' / L])$ to $\Phi_{k, \frakv}$.

\begin{thm}\label{thm:ThetaLiftNonHolomorphicSurjective}
The theta lifts $\Phi_{k, \beta}(Z, s)$ generate the space spanned by the Eisenstein series $\calG_{\kappa, \delta, \chi}$ and hence the space spanned by the Eisenstein series $\calG_{\kappa, \delta}$. In particular, if $\pi_L$ is injective, then the theta lift is surjective onto non-holomorphic Eisenstein series corresponding to $0$-dimensional cusps and if $\pi_L$ is surjective, the theta lift is injective on vector-valued non-holomorphic Eisenstein series.
\end{thm}

\begin{proof}
We first resort the sum. Instead of summing over $\delta \in L' / L$ with $q(\delta) = 0$ and $\beta = k_{\delta \beta} \delta$, we sum over all cyclic isotropic subgroups containing $\beta$ and then the generators of the subgroup. This yields
\begin{align*}
\Phi_{k, \beta}(Z, s)
&= \frac{\Gamma(s + \kappa)}{(-2 i)^{\kappa}\pi^{s + \kappa}} \sum_{\substack{\beta \in H \\ H \text{ isotropic}}} \sum_{\langle \delta \rangle = H} N_\delta^{2s + \kappa} \zeta_+^{k_{\delta \beta}}(2s + \kappa) \calG_{\kappa, \delta}(Z, s) \\
&= \frac{\Gamma(s + \kappa)}{(-2\pi i)^{\kappa}\pi^{s}} \sum_{\substack{\beta \in H = \langle \delta \rangle \\ H \text{ isotropic}}} N_\delta^{2s + \kappa} \sum_{\substack{k \in (\IZ / N_\delta \IZ)^\times}} \zeta_+^{k_{k\delta \beta}}(2s + \kappa) \calG_{\kappa, k\delta}(Z, s) \\
&= \frac{\Gamma(s + \kappa)}{(-2\pi i)^{\kappa}\pi^{s}} \sum_{\substack{\beta \in H = \langle \delta \rangle \\ H \text{ isotropic}}} N_\delta^{2s + \kappa} \sum_{\substack{k \in (\IZ / N_\delta \IZ)^\times}} \zeta_+^{k^* k_{\delta \beta}}(2s + \kappa) \calG_{\kappa, k\delta}(Z, s).
\end{align*}
Let $\chi$ be a Dirichlet character of modulus $N_\beta$ and recall that
$$E_{k, \beta, \chi} := \sum_{m \in (\IZ / N_\beta \IZ)^\times} \chi(m) E_{k, m\beta}.$$
Its lift is obviously given by
$$\frac{\Gamma(s + \kappa)}{(-2\pi i)^{\kappa}\pi^{s}} \sum_{m \in (\IZ / N_\beta \IZ)^\times} \chi(m) \sum_{\substack{m\beta \in H = \langle \delta \rangle \\ H \text{ isotropic}}} N_\delta^{2s + \kappa} \sum_{\substack{k \in (\IZ / N_\delta \IZ)^\times}} \zeta_+^{k^* k_{\delta m \beta}}(2s + \kappa) \calG_{\kappa, k \delta}(Z, s).$$
Now for $m \in (\IZ / N_\beta \IZ)^\times$ we have $m \beta \in H$ if and only if $\beta \in H$. Moreover, we have $k^* k_{\delta m \beta} = k_{\delta k^* m \beta}$ and thus obtain that $\frac{(-2 \pi i)^\kappa \pi^s}{\Gamma(s + \kappa)} \Phi(Z, E_{k, \beta, \chi}(\cdot, s))$ is equal to
\begin{align*}
&\sum_{\substack{\beta \in H = \langle \delta \rangle \\ H \text{ isotropic}}} N_\delta^{2s + \kappa} \sum_{\substack{k \in (\IZ / N_\delta \IZ)^\times}} \calG_{\kappa, k \delta}(Z, s) \sum_{m \in (\IZ / N_\beta \IZ)^\times} \chi(m) \zeta_+^{k^* k_{\delta m\beta}}(2s + \kappa) \\
&= \sum_{\substack{\beta \in H = \langle \delta \rangle \\ H \text{ isotropic}}} N_\delta^{2s + \kappa} \sum_{\substack{k \in (\IZ / N_\delta \IZ)^\times}} \chi(k) \calG_{\kappa, k \delta}(Z, s) \sum_{m \in (\IZ / N_\beta \IZ)^\times} \chi(m) \zeta_+^{k_{\delta m\beta}}(2s + \kappa) \\
&= \sum_{\substack{\beta \in H = \langle \delta \rangle \\ H \text{ isotropic}}} N_\delta^{2s + \kappa} \sum_{m \in (\IZ / N_\beta \IZ)^\times} \chi(m) \zeta_+^{k_{\delta m\beta}}(2s + \kappa) \calG_{\kappa, \delta, \chi}(Z, s).
\end{align*}
Observe that if $\langle \beta \rangle$ is a maximal cyclic isotropic subgroup, then
\begin{align*}
\Phi(Z, E_{k, \beta, \chi})
&= \frac{\Gamma(s + \kappa) N_\beta^{2s + \kappa}}{(-2\pi i)^{\kappa}\pi^{s}} \sum_{m \in (\IZ / N_\beta \IZ)^\times} \chi(m) \zeta_+^{m}(2s + \kappa) \calG_{\kappa, \beta, \chi}(Z, s) \\
&= \frac{\Gamma(s + \kappa) N_\beta^{2s + \kappa}}{(-2\pi i)^{\kappa}\pi^{s}} L(2s + \kappa, \chi) \calG_{\kappa, \beta, \chi}(Z, s).
\end{align*}
An inductive argument now shows the result.
\end{proof}

An immediate consequence is

\begin{cor}\label{cor:EisensteinSeriesFunctionalEquation}
The non-holomorphic Eisenstein series $\calG_{\kappa, \delta}(Z, s)$ for $\delta \in \Iso(L' / L)$ have functional equations relating the values at $s$ with the values at $1 - k - s$.
\end{cor}

We now turn to the Fourier expansion. We will need the following lemma. For similar formulas see \cite{OSullivan}.

\begin{lem}\label{lem:SumOfGummaFunctions}
For $\kappa \in \IZ, \kappa \geq 0$ we have
\begin{align*}
\sum_{j = 0}^{\lfloor \frac{\kappa}{2} \rfloor} (-1)^j {\kappa \choose 2j} \Gamma(1/2 + j) \Gamma(1/2 + s - j)
&= 2^{\kappa-2s} \pi  \frac{\Gamma(1 + 2s - \kappa)}{\Gamma(1 + s - \kappa)} \\
&= \sqrt{\pi} \frac{\Gamma(\frac{1}{2} + s - \frac{\kappa}{2}) \Gamma(1 + s - \frac{\kappa}{2})}{\Gamma(1 + s - \kappa)}.
\end{align*}
\end{lem}

\begin{proof}
We first assume that $s \geq \kappa$ is an integer. Then the duplication formula yields
\begin{align*}
&\sum_{j = 0}^{\lfloor \frac{\kappa}{2} \rfloor} (-1)^j {\kappa \choose 2j} \Gamma(1/2 + j) \Gamma(1/2 + (s - j)) \\
&= \frac{\pi \kappa!}{4^s} \sum_{j = 0}^{\lfloor \frac{\kappa}{2} \rfloor} (-1)^j \frac{1}{(2j)! (\kappa - 2j)!} \frac{(2j)!}{j!} \frac{(2s - 2j)!}{(s - j)!} \\
&= \frac{\pi \kappa! (2s - \kappa)!}{4^s s!} \sum_{j = 0}^{\lfloor \frac{\kappa}{2} \rfloor} (-1)^j {s \choose s - j} {2s - 2j \choose 2s - \kappa}  \\
&= (-1)^{\lfloor \frac{\kappa}{2} \rfloor} \frac{\pi \kappa! (2s - \kappa)!}{4^s s!} \sum_{j = 0}^{\lfloor \frac{\kappa}{2} \rfloor} (-1)^j {s \choose s - \lfloor \frac{\kappa}{2} \rfloor + j} {2s - 2 \lfloor \frac{\kappa}{2} \rfloor + 2j \choose 2s - \kappa} \\
&= (-1)^{s}\frac{\pi \kappa! (2s - \kappa)!}{4^s s!} \sum_{j = s - \lfloor \frac{\kappa}{2} \rfloor}^{s} (-1)^j {s \choose j} {2j \choose 2s - \kappa}  \\
&= \frac{\pi \kappa! (2s - \kappa)!}{2^{2s - \kappa} s!} {s \choose {s - \kappa}} \\
&= \frac{\pi (2s - \kappa)!}{2^{2s - \kappa} (s - \kappa)!} = 2^{\kappa - 2s} \pi \frac{\Gamma(2s - \kappa + 1)}{\Gamma(s - \kappa + 1)}
\end{align*}
This shows the identity for integers $s \geq \kappa$. Now observe using $\Gamma(z+1) = z\Gamma(z)$ and the duplication formula again that
\begin{align*}
2^{\kappa-2s} \pi  \frac{\Gamma(1 + 2s - \kappa)}{\Gamma(1 + s - \kappa)}
&= \sqrt{\pi} \frac{\Gamma(\frac{1}{2} + s - \frac{\kappa}{2}) \Gamma(1 + s - \frac{\kappa}{2})}{\Gamma(1 + s - \kappa)} \\
&= \Gamma(1/2 + s - \kappa) p(s)
\end{align*}
for some polynomial $p$. Similarly, the left hand side is given by
$$\sum_{j = 0}^{\lfloor \frac{\kappa}{2} \rfloor} (-1)^j {\kappa \choose 2j} \Gamma(1/2 + j) \Gamma(1/2 + s - j) = \Gamma(1/2 + s - \kappa) q(s)$$
for some polynomial $q$. Now we have seen that $p(s) = q(s)$ for infinitely many integers and thus $p(s) = q(s)$ for all $s \in \IC$.
\end{proof}

We have
$$\zeta^c(s) = \zeta_+^c(s) + e^{- \pi i s} \zeta_+^{-c}(s).$$
For positive integers $\kappa$ this yields
$$\zeta^c(\kappa) = \zeta_+^c(\kappa) + (-1)^\kappa \zeta_+^{-c}(\kappa)$$
and $\zeta^{-c}(\kappa) = (-1)^\kappa \zeta^c(\kappa)$. Moreover, for $b \in \IZ / N \IZ$ we have the functional equation
$$\sum_{c \in \IZ / N \IZ} e\left(\frac{bc}{N}\right) \zeta_+^c(1 - 2s - \kappa) = \frac{\Gamma(2s + \kappa) N^{2s + \kappa}}{(-2 \pi i)^{\kappa}(2 \pi)^{2s}} e^{\pi i s} \zeta^b(2s + \kappa).$$
The reflection formula implies
$$\frac{\Gamma(1 - 2s - \kappa) \Gamma(2s + \kappa)}{\Gamma(1 - s - \kappa)} = \Gamma(s + \kappa) \frac{\sin(\pi s)}{\sin(2 \pi s)} = \frac{\Gamma(s + \kappa)}{2 \cos(\pi s)}.$$
Together with
$$\zeta^b(2s + \kappa)  + (-1)^\kappa \zeta^{-b}(2s + \kappa) = 2 \cos(\pi s)e^{- \pi i s} (\zeta_+^b(2s + \kappa)  + (-1)^\kappa \zeta_+^{-b}(2s + \kappa))$$
we obtain
\begin{lem}\label{lem:ConstantTermZetaIdentity}
For $\beta \in \IZ / N \IZ$ we have
\begin{align*}
&\frac{\Gamma(1 - 2s - \kappa)}{\Gamma(1 - s - \kappa)} 2^{2s} \pi^s \sum_{b, c \in \IZ / N \IZ} e\left(\frac{bc}{N}\right) (\delta_{\beta, b} + (-1)^\kappa \delta_{-\beta, b}) \zeta_+^c(1 - 2s - \kappa) \\
&= \frac{N^{2s + \kappa} \Gamma(s + \kappa)}{(-2\pi i)^\kappa \pi^{s}} (\zeta_+^\beta(2s + \kappa) + (-1)^\kappa \zeta_+^{-\beta}(2s + \kappa)).
\end{align*}
\end{lem}

The results of Section \ref{sec:UnfoldingAgainstTheta} yield the Fourier expansion

\begin{thm}\label{thm:ThetaLiftFourierExpansion}
Let $z \in \Iso_0(L)$ of level $N_z$ and let $z' \in L'$ with $(z, z') = 1$. The theta lift $\Phi_{k, \beta}(Z, s)$ has the Fourier expansion in the cusp $z$ given by
\begin{align*}
&\frac{i^\kappa }{2 \sqrt{2}\lvert Y \rvert^{\kappa-1}} \Phi^K_{k, \beta}\left(\frac{Y}{\lvert Y \rvert}, s\right) + \sum_{\lambda \in K'} b_{k, \beta}(\lambda, Y, s) e(\lambda, X),
\end{align*}
where the Fourier coefficient $b_{k, \beta}(0, Y, s)$ is given by
\begin{align*}
& \sum_{b \in \IZ / N_z \IZ} \bigg(\frac{\Gamma(s + \kappa) N_z^{2s + \kappa}}{(-2 \pi i)^\kappa \pi^s} q(Y)^s (\delta_{\beta, \frac{bz}{N_z}} + (-1)^\kappa \delta_{- \beta, \frac{bz}{N_z}}) \zeta_+^b(2s + \kappa) \\
&+ \frac{\Gamma(1 - s - k + \kappa) N_z^{2 - 2s - 2k + \kappa}}{(-2 \pi i)^\kappa \pi^{1 - s - k}} q(Y)^{1 - s - k} c_{k, \beta}\left(\frac{bz}{N_z}, 0, s\right) \zeta_+^b(2 - 2s - 2k + \kappa)\bigg).
\end{align*}
For $q(\lambda) = 0, \lambda \neq 0$ the Fourier coefficient $b_{k, \beta}(\lambda, Y, s)$ is given by
\begin{align*}
& \frac{2 \lvert(\lambda, Y) \rvert^{\frac{1}{2}}}{2^\kappa} \sum_{b \in \IZ / N_z \IZ} \bigg(\frac{q(Y)^s}{\lvert (\lambda, Y) \rvert^{s}} e\left(- \frac{(\lambda, \zeta)}{N_z}\right) \\
&\times \sum_{n \mid \lambda} n^{2s - 1 + \kappa} (\delta_{\beta, \frac{\lambda}{n} - \frac{(\lambda, \zeta)}{nN_z} z + \frac{bz}{N_z}} + (-1)^\kappa \delta_{-\beta, \frac{\lambda}{n} - \frac{(\lambda, \zeta)}{nN_z} z + \frac{bz}{N_z}}) e\left(\frac{nb}{N_z}\right) \\
&\times \sum_{h = 0}^\infty \sum_{j = 0}^\infty \frac{(-1)^{j}}{(4 \pi \lvert(\lambda, Y)\rvert)^j j!} {\kappa \choose h} \frac{(\kappa - h)!}{(\kappa-h-2j)!} \\
&\times \left(\frac{(\lambda, Y)}{\lvert (\lambda, Y) \rvert}\right)^{\kappa-h} K_{s - \frac{1}{2} + \kappa - h - j}(2 \pi \lvert (\lambda, Y) \rvert)\\
&+ \frac{q(Y)^{1 - s - k}}{\lvert (\lambda, Y) \rvert^{1 - s - k}} e\left(- \frac{(\lambda, \zeta)}{N_z}\right) \\
&\times \sum_{n \mid \lambda} n^{1 - 2s - 2k + \kappa} c_{k, \beta}\left(\frac{\lambda}{n} - \frac{(\lambda, \zeta)}{nN_z} z + \frac{bz}{N_z}, 0, s\right) e\left(\frac{nb}{N_z}\right) \\
&\times \sum_{h = 0}^\infty \sum_{j = 0}^\infty \frac{(-1)^{j}}{(4 \pi \lvert(\lambda, Y)\rvert)^j j!} {\kappa \choose h} \frac{(\kappa - h)!}{(\kappa-h-2j)!} \\
&\times \left(\frac{(\lambda, Y)}{\lvert (\lambda, Y) \rvert}\right)^{\kappa-h} K_{\frac{1}{2} - s - k + \kappa - h - j}(2 \pi \lvert (\lambda, Y ) \rvert)\bigg).
\end{align*}
For $q(\lambda) \neq 0$ the Fourier coefficient $b_{k, \beta}(\lambda, Y, s)$ is given by
\begin{align*}
&\frac{1}{\sqrt{2}\lvert Y \rvert^{\kappa-1}} \sum_{b \in \IZ / N_z \IZ} \sum_{h = 0}^\infty (2i)^{-h} \sum_{j = 0}^\infty \frac{(-1)^j i^{h}}{(8 \pi)^j j!} {\kappa \choose h} \frac{(\kappa-h)!}{(\kappa-h-2j)!} \lvert Y \rvert^h \left(\frac{(\lambda, Y)}{\lvert Y \rvert}\right)^{\kappa-h - 2j} \\
&\times e\left(- \frac{(\lambda, \zeta)}{N_z}\right) \sum_{n \mid \lambda} n^{2j - \kappa + 2h} e\left(\frac{nb}{N_z}\right) c_{k, \beta}\left(\frac{\lambda}{n} - \frac{(\lambda, \zeta)}{nN_z} z + \frac{bz}{N_z}, \frac{q(\lambda)}{n^2}, s\right) \\
&\times \int_0^\infty \exp\left(-\frac{\pi n^2}{2vz_Z^2} - \frac{2 \pi v q_w(\lambda)}{n^2}\right) \calW_s\left(4 \pi \frac{q(\lambda)}{n^2}v\right) v^{-\frac{3}{2} + \kappa - h - j} \mathrm{d}v.
\end{align*}
\end{thm}

\begin{proof}
Theorem \ref{thm:UnfoldingAgainstTheta} shows that the theta lift $\Phi_{k, \beta}(Z, s)$ is given by
\begin{align*}
&\frac{i^\kappa \lvert Y \rvert^{1 - \kappa}}{2\sqrt{2}} \Phi^K_{k, \beta}(w, p_{\omega,0}, s) + \frac{i^\kappa \lvert Y \rvert^{1 - \kappa}}{\sqrt{2}} \sum_{h} \frac{1}{(2 i)^{h}} \\
&\times \sum_{\lambda \in K'} \sum_{\substack{\delta \in L_0' / L \\ \pi(\delta) = \lambda + K}} \sum_{j = 0}^{\infty} \frac{(-\Delta)^j(\overline{p}_{\omega,h})(w(\lambda))}{(8\pi)^j j!} \sum_{n = 1}^\infty n^{h} e(n((\lambda, X) + (\delta, z'))) \\
&\times \int_{v = 0}^\infty \exp\left(-\frac{\pi n^2 q(Y)}{v} - 2 \pi v q_w(\lambda)\right) c_{k, \beta}(\delta, q(\lambda), s, v) v^{-\frac{3}{2} + \kappa - h - j - t} \mathrm{d}v.
\end{align*}
We have
\begin{align*}
p(\nu(\lambda))
&= \left(\frac{(\lambda, X + iY)}{\lvert Y \rvert}\right)^\kappa \\
&= \sum_{h} (\lambda, X)^{h} {\kappa \choose h} i^{\kappa - h} \lvert Y \rvert^{-\kappa} (\lambda, Y)^{\kappa - h} \\
&= \sum_{h} (\lambda, z_{\nu^+})^{h} {\kappa \choose h} i^{\kappa - h} \lvert Y \rvert^{2h - \kappa} (\lambda, Y)^{\kappa - h},
\end{align*}
hence
\begin{align*}
p_{\omega, h}(w(\lambda)) = {\kappa \choose h} i^{\kappa - h} \lvert Y \rvert^{h} (\lambda, Y / \lvert Y \rvert)^{\kappa - h}
\end{align*}
and
\begin{align*}
(-\Delta)^j \overline{p}_{\omega, h}(w(\lambda)) 
&= (-1)^j i^{h - \kappa} {\kappa \choose h} \frac{(\kappa - h)!}{(\kappa - h - 2j)!} \lvert Y \rvert^{h} (\lambda, Y / \lvert Y \rvert)^{\kappa - h - 2j}.
\end{align*}
The term for $\lambda = 0$ is now given by
\begin{align*}
&\frac{i^\kappa \lvert Y \rvert^{1 - \kappa}}{\sqrt{2}} \sum_{h} \frac{1}{(2 i)^{h}} \sum_{j = 0}^{\infty} \frac{(-\Delta)^j(\overline{p}_{\omega,h})(0)}{(8\pi)^j j!} \left(\pi q(Y)\right)^{\kappa - h - j - \frac{1}{2}} \\
&\times \sum_{b, c \in \IZ / N_z \IZ} e\left(\frac{bc}{N_z}\right) \bigg(\Gamma\left(\frac{1}{2} - s - \kappa + h + j\right) \left(\pi q(Y)\right)^{s} \\
&\times (\delta_{\beta, \frac{bz}{N_z}} + (-1)^\kappa \delta_{- \beta, \frac{bz}{N_z}}) \zeta_+^c(-2s + 1 - 2\kappa + h + 2j) \\
&+ \Gamma\left(-\frac{1}{2} - \kappa + h + j + k + s\right) \left(\pi q(Y)\right)^{1 - k - s} \\
&\times c_{k, \beta}\left(\frac{bz}{N_z}, 0, s\right) \zeta_+^c(- 1 - 2\kappa + h + 2j + 2k + 2s)\bigg).
\end{align*}
Obviously
$$(-\Delta)^j(\overline{p}_{\omega,\kappa - 2j})(0) = {\kappa \choose 2j} (2j)! \lvert Y \rvert^{\kappa - 2j}$$
and $(-\Delta)^j(\overline{p}_{\omega,h})(0) = 0$ for $\kappa - 2j \neq h$. Hence we can rewrite this to
\begin{align*}
&\frac{1}{\sqrt{\pi}2^{\kappa}} \sum_{j = 0}^{\infty} (-1)^j {\kappa \choose 2j} \frac{(2j)!}{4^j j!} \\
&\times \sum_{b, c \in \IZ / N_z \IZ} e\left(\frac{bc}{N_z}\right) \bigg(\Gamma\left(\frac{1}{2} - s - j\right) \left(\pi q(Y)\right)^{s} (\delta_{\beta, \frac{bz}{N_z}} + (-1)^\kappa \delta_{- \beta, \frac{bz}{N_z}}) \zeta_+^c(1 - 2s - \kappa) \\
&+ \Gamma\left(s - \frac{1}{2} - j + k\right) \left(\pi q(Y)\right)^{1 - k - s} c_{k, \beta}\left(\frac{bz}{N_z}, 0, s\right) \zeta_+^c(2s - 1 - \kappa + 2k)\bigg).
\end{align*}
Using the duplication formula
\begin{align*}
\frac{(2j)!}{4^j j!} = \frac{\Gamma(j + \frac{1}{2})}{\sqrt{\pi}}
\end{align*}
and Lemma \ref{lem:SumOfGummaFunctions} we obtain
\begin{align*}
&\sum_{b, c \in \IZ / N_z \IZ} e\left(\frac{bc}{N_z}\right) \bigg(\frac{\Gamma(1 - 2s - \kappa)}{\Gamma(1 - s - \kappa)} 4^s \pi^{s} q(Y)^s (\delta_{\beta, \frac{bz}{N_z}} + (-1)^\kappa \delta_{- \beta, \frac{bz}{N_z}}) \zeta_+^c(1 - 2s - \kappa) \\
&+ \frac{\Gamma(2s + 2k - 1 - \kappa)}{\Gamma(s + k - \kappa)} 4^{1 - s - k} \pi^{1 - s - k} q(Y)^{1 - s - k} c_{k, \beta}\left(\frac{bz}{N_z}, 0, s\right) \zeta_+^c(2s + 2k - 1 - \kappa)\bigg).
\end{align*}
Applying Lemma \ref{lem:ConstantTermZetaIdentity} yields
\begin{align*}
&\sum_{b \in \IZ / N_z \IZ} \bigg(\frac{N_z^{2s + \kappa} \Gamma(s + \kappa)}{(-2 \pi i)^\kappa \pi^s} q(Y)^s (\delta_{\beta, \frac{bz}{N_z}} + (-1)^\kappa \delta_{- \beta, \frac{bz}{N_z}}) \zeta_+^b(2s + \kappa) \\
&+ \frac{N_z^{2 - 2s - 2k + \kappa} \Gamma(1 - s - k + \kappa)}{(-2 \pi i)^\kappa \pi^{1 - s - k}} q(Y)^{1 - s - k} c_{k, \beta}\left(\frac{bz}{N_z}, 0, s\right) \zeta_+^b(2 - 2s - 2k + \kappa)\bigg).
\end{align*}
Now, the terms for $q(\lambda) = 0, \lambda \neq 0$ are given by
\begin{align*}
&\frac{i^\kappa \lvert Y \rvert^{1 - \kappa}}{\sqrt{2}} \sum_{h} \frac{1}{(2 i)^{h}} \sum_{\substack{\delta \in L_0' / L \\ \pi(\delta) = \lambda + K}} \sum_{j = 0}^{\infty} \frac{(-\Delta)^j(\overline{p}_{\omega,h})(w(\lambda))}{(8\pi)^j j!} \sum_{n = 1}^\infty n^{h} e(n((\lambda, X) + (\delta, z'))) \\
&\times \int_{v = 0}^\infty \exp\left(-\frac{\pi n^2 q(Y)}{v} - 2 \pi v q_w(\lambda)\right) c_{k, \beta}(\delta, q(\lambda), s, v) v^{-\frac{3}{2} + \kappa - h - j - t} \mathrm{d}v
\end{align*}
and the integral is
\begin{align*}
& 2 (\delta_{\beta, \delta} + (-1)^\kappa \delta_{-\beta, \delta}) \left(\frac{n q(Y)}{\lvert (\lambda, Y) \rvert}\right)^{s - \frac{1}{2} + \kappa - h - j - t} K_{s - \frac{1}{2} + \kappa - h - j - t}\left(2\pi n \lvert (\lambda, Y) \rvert\right) \\
&+ 2 c_{k, \beta}(\delta, 0, s) \left(\frac{n q(Y)}{\lvert (\lambda, Y) \rvert}\right)^{\frac{1}{2} + \kappa - h - j - k - s - t} K_{\frac{1}{2} + \kappa - h - j - k - s - t}\left(2\pi n \lvert (\lambda, Y) \rvert \right).
\end{align*}
Now plug in the definition of $(- \Delta)^j (\overline{p}_{\omega, h})(w(\lambda))$ and reorder the sum as a divisor sum to obtain for $t = 0$
\begin{align*}
& 2^{- \kappa} \sum_{h} \sum_{j = 0}^{\infty} \frac{(-1)^j}{(4\pi \lvert (\lambda, Y) \rvert)^j j!} {\kappa \choose h} \frac{(\kappa - h)!}{(\kappa - h)!} \left(\frac{(\lambda, Y)}{\lvert (\lambda, Y) \rvert} \right)^{\kappa - h} \\
&\times \sum_{n \mid \lambda}^\infty \sum_{\substack{\delta \in L_0' / L \\ \pi(\delta) = \frac{\lambda}{n} + K}} e((\lambda, X) + (n\delta, z')) \\
&\times \bigg(2 (\delta_{\beta, \delta} + (-1)^\kappa \delta_{-\beta, \delta}) n^{2s - 1 + \kappa} \lvert (\lambda, Y) \rvert^{\frac{1}{2} - s} q(Y)^{s} K_{s - \frac{1}{2} + \kappa - h - j}\left(2\pi \lvert (\lambda, Y) \rvert\right) \\
&+ 2 c_{k, \beta}(\delta, 0, s) n^{1 - \kappa - 2k - 2s} \lvert (\lambda, Y) \rvert^{s + k - \frac{1}{2}} q(Y)^{1 - k - s} K_{\frac{1}{2} + \kappa - h - j - k - s}\left(2\pi \lvert (\lambda, Y) \rvert \right)\bigg).
\end{align*}
Now $\delta \in L_0' / L$ with $\pi(\delta) = \frac{\lambda}{n} + K$ are given by $\frac{\lambda}{n} - \frac{(\lambda, \zeta)}{nN_z} z + \frac{b}{N_z} z$, where $b$ runs through $\IZ / N_z \IZ$.
This shows the result.
\end{proof}

Denote the Fourier coefficient of $q(Y)^{1 - k - s}$ by $\varphi_{k, \beta}\left(\frac{z}{N_z}, s\right)$, i.e. it is given by
$$\frac{\Gamma(1 - s - k + \kappa) N_z^{2 - 2s - 2k + \kappa}}{(-2 \pi i)^\kappa \pi^{1 - s - k}} \sum_{b \in \IZ / N_z \IZ} c_{k, \beta}\left(\frac{bz}{N_z}, 0, s\right) \zeta_+^b(2 - 2s - 2k + \kappa).$$
Then, as we would expect, the functional equation has the form
\begin{align*}
\Phi_{k, \beta}(Z, s)
&= \frac{\Gamma(s + \kappa)}{(-2\pi i)^{\kappa}\pi^{s}} \sum_{\delta \in \Iso(L' / L)} N_\delta^{2s + \kappa} \zeta_+^{k_{\delta\beta}}(2s + \kappa) \calG_{\kappa, \delta}(Z, s) \\
&= \frac{1}{2} \sum_{\delta \in \Iso(L' / L)} \varphi_{k, \beta}(\delta, s) \calG_{\kappa, \delta}(Z, 1 - s - k).
\end{align*}

\begin{exmpl}
Assume that $L$ is a maximal lattice of Witt rank $2$. Then it can be easily seen that $L$ splits two hyperbolic planes over $\IZ$. Up to the action of $\Gamma(L)$ there is only one cusp $z \in \Iso_0(L)$ and the discriminant group $L' / L$ is anisotropic, i.e. $\Iso(L' / L) = \{0\}$. Hence there is exactly one vector-valued non-holomorphic Eisenstein series $E_{k, 0}(\tau, s)$ and exactly one orthogonal non-holomorphic Eisenstein series $\calE_{\kappa, z}(Z, s)$ for $\Gamma(L)$ corresponding to $0$-dimensional cusps. We obtain the meromorphic continuation of $\calE_{\kappa, z}(Z, s)$ and
$$\Phi_{k,0}(Z, s) = 2 \frac{\Gamma(s + \kappa) \zeta(2s + \kappa)}{(-2 \pi i)^\kappa \pi^s} \calE_{\kappa, z}(Z, s).$$
Moreover, we have the functional equation
$$\calE_{\kappa, z}(Z, s) = \frac{1}{2} \tilde{\varphi}_{\kappa, 0}(z, s) \calE_{\kappa, z}(Z, 1 - k - s),$$
where $\tilde{\varphi}_{k, 0}(z, s)$ is given by
$$\tilde{\varphi}_{k, 0}(z, s) = \frac{\pi^s \Gamma(1 - k - s + \kappa) \zeta(2(1 - k - s) + \kappa)}{\pi^{1 - k - s} \Gamma(s + \kappa) \zeta(2s + \kappa)} c_{k, 0}(0, 0, s).$$
\end{exmpl}

\renewcommand\bibname{References}
\bibliographystyle{alphadin}
\bibliography{OrthogonalEisensteinSeriesAndThetaLifts}

\end{document}